\documentclass[12pt,a4paper]{amsproc}

\usepackage{amsmath}
\usepackage{amsfonts}
\usepackage{amssymb}

\usepackage[dvips,colorlinks,bookmarks,breaklinks]{hyperref} 
\usepackage{graphicx}       
\usepackage[T1]{fontenc}
\usepackage{lmodern}        

\newcommand{\change}[1][]{}

\newcommand{\bem}[1]{}

\allowdisplaybreaks[1]

\newtheorem{theorem}{Theorem}[section]
\newtheorem{lemma}[theorem]{Lemma}
\newtheorem{proposition}[theorem]{Proposition}
\newtheorem{corollary}[theorem]{Corollary}

\theoremstyle{definition}
\newtheorem{definition}[theorem]{Definition}

\theoremstyle{remark}
\newtheorem{remark}[theorem]{Remark}

\numberwithin{equation}{subsection}
\numberwithin{theorem}{subsection}

\newcommand{\id}{\ensuremath{\mathbf{1}}}

\newcommand{\cN}{\ensuremath{\mathcal{N}}}

\title[Symmetries of the transfer operator for $\Gamma_0(N)$]{Symmetries of the transfer operator for $\Gamma_0(N)$ and a character deformation of the Selberg zeta function for $\Gamma_0(4)$}

\author{M.~Fraczek}
\address{\href{http://www.dynamik.tu-clausthal.de}{Institute of Theoretical Physics}, \href{http://www.tu-clausthal.de}{TU Clausthal}, 38678- Clausthal-Zellerfeld, Germany}
\email{\href{mailto:markus.fraczek@tu-clausthal.de}{markus.fraczek@tu-clausthal.de}}

\author{D.~Mayer}
\address{\href{http://www.dynamik.tu-clausthal.de}{Lower Saxony Professorship}, \href{http://www.tu-clausthal.de}{TU Clausthal}, 38678- Clausthal-Zellerfeld, Germany}
\email{\href{mailto:dieter.mayer@tu-clausthal.de}{dieter.mayer@tu-clausthal.de}}
\thanks{This work was supported  by the Deutsche Forschungsgemeinschaft through the DFG Research Project ``Maass wave forms and the transfer operator approach to the Phillips-Sarnak conjecture'' (Ma 633/18-1).}


\date{\today}


\begin{document}

\maketitle

\begin{abstract}
The transfer operator  for $\Gamma_0(N)$ and  trivial character $\chi_0$ possesses  a finite group 
of symmetries generated by permutation matrices $P$ with $P^2=id$. Every such  symmetry  leads to a factorization of the Selberg zeta function in terms of  Fredholm determinants of a reduced transfer operator. These symmetries are related to the group of automorphisms in $GL(2,\mathbb{Z})$ of the Maass wave forms  of $\Gamma_0(N)$ .  For the group $\Gamma_0(4)$ and Selberg's character $\chi_\alpha$ there exists just one  non-trivial symmetry operator $P$. The eigenfunctions  of the corresponding reduced transfer operator with eigenvalue $\lambda=\pm 1$  are related to Maass forms even respectively odd under a corresponding  automorphism. It  then follows from a result of Sarnak and Phillips that the zeros of the Selberg function determined by the eigenvalues $\lambda=-1$ of the reduced transfer operator stay on the critical line under the deformation of the character. From numerical results we expect that on the other hand all the zeros corresponding to the eigenvalue $\lambda=+1$ leave this line  for $\alpha$ turning away from zero.
\end{abstract}

\setcounter{tocdepth}{3} \tableofcontents

\section{Introduction}
In the transfer operator approach to Selberg's zeta function for a Fuchsian group $\Gamma$ this function gets expressed 
in terms of the Fredholm determinant of this operator which is constructed from the symbolic dynamics of the geodesic 
flow on the corresponding surface of constant negative curvature.  Even if this approach has been carried out up to now 
only for certain groups like  the modular subgroups of finite index \cite{CM00},\cite{CM01},\cite{CM01a}, or  the Hecke triangle groups 
\cite{MS08}, \cite{MM10},\cite{MMS11} it has lead for instance to new points of view on this function \cite{Z02} or  the theory 
of period functions \cite{LZ01}. Another application of this method is a precise numerical calculation of the 
Selberg zeta function \cite{St08}, which seems to be impossible by other means at the 
moment. In this paper we discuss the transfer operator approach to  Selberg's zeta function for Hecke congruence 
subgroups with character,  of special interest being the behaviour of its zeros for $\Gamma_0(4)$ under the singular 
deformation of Selberg's  character \cite{Se89}.\newline
  As found numerically by M. Fraczek in \cite{F10},   certain symmetries of the transfer operator for these 
groups play thereby an 
important role. These symmetries lead to a 
factorization of the Selberg zeta function as known for the full modular group $SL(2,\mathbb{Z})$. There
it corresponds to the involution $Ju(z)=u(-z^*)$ of the Maass forms $u$ for this group \cite{E93}, \cite{LZ01}. Obviously the 
corresponding element $j=\begin{pmatrix}1&0\\0&-1\end{pmatrix}\in GL(2,\mathbb{Z})$ generates the normalizer 
group of $SL(2,\mathbb{Z})$ in $GL(2,\mathbb{Z})$. It tuns out that also the symmetries of the transfer operator for 
$\Gamma_0(N)$ correspond to automorphisms of the Maass forms from its  normalizer group  in $GL(2,\mathbb{Z})$. \newline
 For the group $\Gamma_0(4)$ with a character $\chi_\alpha$ introduced by Selberg in
\cite{Se89}  and discussed also by Phillips and Sarnak in \cite{PS85}, there is only one such non-trivial symmetry of the 
transfer operator. It corresponds to the generator of $\Gamma_0(4)$'s normalizer group in $GL(2,\mathbb{Z})$  leaving invariant the character $\chi_\alpha$. Results of Sarnak and Phillips  imply that the zeros on the critical line of one factor of  Selberg's function stay on this line under the 
deformation of the character, and hence the corresponding  Maass wave forms for the trivial character remain Maass 
wave forms. Numerical results \cite{F10} on the other hand imply, that the zeros on the critical line of the second factor of 
this function should all leave this line when the deformation is turned on. A detailed discussion of these numerical results 
and their partial proofs is in preparation \cite{BFM10}. \newline
The paper is organized as follows: in  Section 2 we recall briefly the form of the transfer operator ${\bf{L}}_{\beta,
\rho_\pi}=\begin{pmatrix}0&\mathcal{L}_{\beta,\pi}^+\\\mathcal{L}_{\beta,\pi}^-&0\end{pmatrix}$ for a general finite index subgroup 
$\Gamma$ of  the modular group $SL(2,\mathbb{Z})$ and unitary representation $\pi$ and introduce the  
symmetries $\tilde{P}=\begin{pmatrix}0&P\\P&0\end{pmatrix}$ of this operator defined by permutation matrices $P$. Any 
such symmetry leads to a factorization of the Selberg zeta function in terms of the Fredholm determinants of the reduced 
transfer operator $P\mathcal{L}_{\beta,\pi}^+$. The eigenfunctions with eigenvalue $\lambda=\pm 1$ of this reduced 
transfer operator then fulfill certain functional equations. In Section 3 we discuss the generators $J_{n,-}$ of the group of automorphisms  in $GL(2,\mathbb{Z})$ of 
the Maass forms $u$ for $\Gamma=\Gamma_0(N)$ and $\pi=\chi_0$  the trivial character. We introduce their  period 
functions $\underline{\psi}$ and derive a formula for the period function $J_{n,-}\underline{\psi}$ of the 
Maass form $J_{n,-}u$. In Section 4 we introduce Selberg's character $\chi_\alpha$ and the non-trivial automorphism 
$J_{2,-}$ of the Maass forms for $\Gamma_0(4)$. We derive again a formula for the period 
function $J_{2,-}\underline{\psi}$ of the Maass form $J_{2,-}u$ leading to a permutation 
matrix $P_{2,-}$ which defines a symmetry $\tilde{P}_{2,-}$ of the transfer operator ${\bf{L}}_{\beta,\rho_{\chi_\alpha}}$.
From this we conclude that the eigenfunctions with eigenvalue $\lambda=\pm 1$ of the operator $P_{2,-}\mathcal{L}_{\beta,\pi}^+$ correspond to Maass forms even respectively odd under the involution $J_{2,-}$. Former results of Phillips and Sarnak then imply that the zero's of the Selberg function on the critical line corresponding to the eigenfunctions with eigenvalue $\lambda =-1$ of this operator  stay on this line under the deformation of the character.

\section{The transfer operator and Selberg's zeta function for   Hecke congruence subgroups $\Gamma_0(N)$} 
The starting point of the transfer operator approach to Selberg's zeta function for a subgroup $\Gamma $ of the modular 
group $SL(2,\mathbb{Z})$ of  index $\mu=[SL(2,\mathbb{Z}):\Gamma]<\infty$ is the geodesic flow 
$\Phi_t:SM_\Gamma\to SM_\Gamma$ on the unit tangent bundle  $SM_\Gamma$ of the  corresponding surface 
$M_\Gamma = \Gamma\setminus \mathbb{H}$ of constant negative curvature. Here $\mathbb{H}=\{z=x+i y: y>0\}$ 
denotes the hyperebolic plane with hyperbolic metric $ds^2=\frac{dx^2+dy^2}{y^2}$ on which the group $\Gamma $ acts 
via M\"obius transformations $gz=\frac{az+b}{cz+d}$ if $g=\begin{pmatrix}a&b\\c&d\end{pmatrix}$. In the present paper 
we are mostly working with the Hecke congruence subgroup $$\Gamma_0(N)=\{g\in SL(2,\mathbb{Z}):g= \begin{pmatrix} 
a&b\\cN&d\end{pmatrix}\}$$ with index 
$\mu_N=N \prod \limits_{p\mid N} (1+\frac{1}{p})$, where $p$ is a prime number. If $\rho:\Gamma\to \text{end}( 
\mathbb{C}^d)$ is a unitary representation of $\Gamma$ then Selberg's zeta function $Z_{\Gamma,\rho}$ is defined as
\begin{equation}\label{Selberg}
Z_{\Gamma,\rho}(\beta)=\prod\limits_\gamma\prod\limits_{k=o}^\infty \det \left(1-\rho(g_\gamma) \exp(-(k+\beta) l_\gamma)\right),\end{equation}
 where $l_\gamma$ denotes the period of the prime periodic orbit $\gamma$
of $\Phi_t$ and $g_\gamma\in \Gamma$ is hyperbolic with $g_\gamma(\gamma)=\gamma$. In the dynamical approach 
to this function it gets expressed in terms of the so called transfer operator well known from D. Ruelle's  
thermodynamic formalism approach to dynamical systems. For general modular groups $\Gamma$ with finite index $\mu$ and 
finite dimensional representation $\pi$ this  operator   ${\bf{L}}_{\beta,\pi}:B\to B$  was determined in \cite{CM00},\cite{CM01} as
 \begin{equation}\label{bftransfer}
{\bf{L}}_{\beta,\pi}=\begin{pmatrix}0&\mathcal{L}_{\beta,\rho_\pi}^+\\ \mathcal{L}_{\beta,
\rho_\pi}^-&0\end{pmatrix},
\end{equation}
 where $B=B(D,\mathbb{C}^{\mu})\bigoplus B(D,\mathbb{C}^{\mu})$ is the Banach space of 
holomorphic functions on the disc $D=\{z: \mid z-1\mid<\frac{3}{2}\}$, and $\rho_\pi$ denotes the  representation of $SL(2,
\mathbb{Z})$ induced from the representation $\pi$ of $\Gamma$ whereas $\mathcal{L}_{\beta,\rho_\pi}^\pm$ is given for $\Re\beta >\frac{1}{2}$ by 
 
\begin{equation}\label{transfer}
\left(\mathcal{L}_{\beta,\rho_\pi}^\pm\underline{f}\right)(z)=\sum\limits_{n=1}^\infty \frac{1}
{(z+n)^{2\beta}}\rho_\pi(ST^{\pm n}) \underline{f}(\frac{1}{z+n}), \end{equation}
where $S=\begin{pmatrix}0&-1\\1&0\end{pmatrix}$ and $T=\begin{pmatrix}1&1\\0&1\end{pmatrix}$.
In the following we restrict ourselves to one dimensional unitary representations $\pi$, hence unitary characters, which we denote as usual by $\chi$. In this case the following Theorem was proved in \cite{CM01}.
\begin{theorem}
 The transfer operator ${\bf{L}}_{\beta,\chi} :B\to B$ with 
 ${\bf{L}}_{\beta,\chi}=\begin{pmatrix}0&\mathcal{L}_{\beta,\chi}^+\\ 
 \mathcal{L}_{\beta,
\chi}^-&0\end{pmatrix}$ and $\left(\mathcal{L}_{\beta,\chi}^\pm\underline{f}\right)(z)=\sum\limits_{n=1}^\infty 
\frac{1}{(z+n)^{2\beta}}\rho_\chi(ST^{\pm n}) \underline{f}(\frac{1}{z+n})$  extends to a meromorphic family of nuclear 
operators of order zero in the entire complex $\beta$ plane with possible poles at $\beta_k=\frac{1-k}{2},\, k=0,1,2,
\ldots$. The Selberg zeta function $Z_{\Gamma,\chi}$ for the modular group $\Gamma$ and character $\chi$ can be 
expressed as $Z_{\Gamma,\chi}(\beta)=\det(1-{\bf{L}}_{\beta,\chi})=\det(1-\mathcal{L}_{\beta,
\chi}^+\mathcal{L}_{\beta,\chi}^-)=\det(1-\mathcal{L}_{\beta,\chi}^-\,\mathcal{L}_{\beta,
\chi}^+)$.
\end{theorem}
This shows that the zero's of Selberg's function are given by those $\beta$-values for which $\lambda=1$ belongs to the spectrum $\sigma ({\bf{L}}_{\beta,\chi})$ respectively $\sigma (\mathcal{L}_{\beta,\chi}^-\,
\mathcal{L}_{\beta,\chi}^+)=\sigma (\mathcal{L}_{\beta,\chi}^+\,\mathcal{L}_{\beta,\chi}^-)$. From Selberg's trace 
formula one knows that there are two kinds of such zeros: the trivial zeros at $\beta=-k, \, k=1,2,\ldots$, and the so 
called spectral zeros. They correspond either to eigenvalues $\lambda=\beta(1-\beta)$ of the automorphic Laplacian with 
$\Re \beta = \frac{1}{2}$ or $\frac{1}{2}\leq \beta\leq 1$ respectively to resonances of the Laplacian, that means poles of 
the scattering determinant with $\Re \beta < \frac{1}{2}$ and $\Im \beta > 0$ \cite{V90}\cite{H83}. For arithmetic groups 
like the congruence subgroups with trivial or congruent character $\chi$ one knows that these resonances are on the line $\Re \beta= \frac{1}{4}$, corresponding to the nontrivial zeros $\zeta_R(2 \beta)=0$ of Riemann's zeta function $\zeta_R$
when assuming his hypothesis, respectively on the line $\Re \beta = 0$. For general Fuchsian groups and congruence 
subgroups with non-congruent character however these resonances can be anywhere in the halfplane $\Re \beta < \frac{1}{2}$.

\subsection{Symmetries of the transfer operator for $\Gamma_0(N)$}
It turns out that there exists for any $N$ a finite number $h_N$ of  $\mu_N\times \mu_N$ permutation matrices $P$ with 
$P^2=id_{\mu_N}$ such that the matrix $\tilde{P}=\begin{pmatrix}0&P\\P&0\end{pmatrix}$ commutes with the transfer operator  
${\bf{L}}_{\beta,\chi}$ and hence
\begin{equation}\label{symmetry}
P\,\mathcal{L}_{\beta,\chi}^+=\mathcal{L}_{\beta,\chi}^-P.
\end{equation} 
Thereby  $P=(P_{i\,j})_{1\leq i,j\leq \mu_N}$ acts in the Banach space $B(D,\mathbb{C}^{\mu_N})$ as $(P\underline{f})_i(z)=\sum\limits_{j=1}^{\mu_N} P_{i\,j} f_j(z)$ if 
$\underline{f}(z)=(f_i(z))_{1\leq i\leq\mu_N}$. We call such a matrix $\tilde{P}$ a symmetry of the transfer operator.  As an example consider the group $\Gamma_0(4)$ and Selberg's character $\chi_\alpha, 0\leq \alpha\leq 1$, which will be described later. Its transfer operator ${\bf{L}}_{\beta,\chi_\alpha}$  has the following form
\begin{eqnarray*}\label{transfer04}
{\bf{L}}_{\beta,\chi_\alpha} \tilde{f}_{+1}&=&\sum\limits_{q=0}^\infty f_{-3}\vert_{2\beta} \tilde{S} T^{1+4q}+f_{-4}\vert_{2\beta}\tilde{S} 
T^{2+4q}+f_{-5}\vert_{2\beta}\tilde{S} T^{3+4q}\\&+&f_{-2}\vert_{2\beta}\tilde{S} T^{4+4q}\\
{\bf{L}}_{\beta,\chi_\alpha} \tilde{f}_{+2}&=&\sum\limits_{q=0}^\infty e^{2\pi i (1+4q)\alpha}f_{-1}\vert_{2\beta} \tilde{S} T^{1+4q}+e^{2\pi i (2+4q)\alpha}f_{-1}\vert_{2\beta}\tilde{S} 
T^{2+4q}\\&+&e^{2\pi i (3+4q)\alpha}f_{-1}\vert_{2\beta}\tilde{S} T^{3+4q}+e^{2\pi i (4+4q)\alpha}f_{-1}\vert_{2\beta}\tilde{S} T^{4+4q}\\
{\bf{L}}_{\beta,\chi_\alpha} \tilde{f}_{+3}&=&\sum\limits_{q=0}^\infty e^{-2\pi i \alpha} f_{-2}\vert_{2\beta} \tilde{S} T^{1+4q}+e^{-2\pi i \alpha}f_{-3}\vert_{2\beta}\tilde{S} 
T^{2+4q}\\&+&e^{-2\pi i \alpha}f_{-4}\vert_{2\beta}\tilde{S} T^{3+4q}+e^{-2\pi i \alpha}f_{-5}\vert_{2\beta}\tilde{S} T^{4+4q}\\
{\bf{L}}_{\beta,\chi_\alpha} \tilde{f}_{+4}&=&\sum\limits_{q=0}^\infty e^{-2 \pi i \alpha(1+4q)}f_{-6}\vert_{2\beta} \tilde{S} T^{1+4q}+e^{-2 \pi i \alpha(2+4q)}f_{-6}\vert_{2\beta}\tilde{S} 
T^{2+4q}\\&+&e^{-2 \pi i \alpha(3+4q)}f_{-6}\vert_{2\beta}\tilde{S} T^{3+4q}+e^{-2 \pi i \alpha(4+4q)}f_{-6}\vert_{2\beta}\tilde{S} T^{4+4q}\\
{\bf{L}}_{\beta,\chi_\alpha} \tilde{f}_{+5}&=&\sum\limits_{q=0}^\infty e^{2\pi i\alpha} f_{-4}\vert_{2\beta} \tilde{S} T^{1+4q}+e^{2\pi i\alpha}f_{-5}\vert_{2\beta}\tilde{S} 
T^{2+4q}\\&+&e^{2\pi i\alpha}f_{-2}\vert_{2\beta}\tilde{S} T^{3+4q}+e^{2\pi i\alpha}f_{-3}\vert_{2\beta}\tilde{S} T^{4+4q}\\
{\bf{L}}_{\beta,\chi_\alpha} \tilde{f}_{+6}&=&\sum\limits_{q=0}^\infty f_{-5}\vert_{2\beta} \tilde{S} T^{1+4q}+f_{-2}\vert_{2\beta}\tilde{S} 
T^{2+4q}+f_{-3}\vert_{2\beta}\tilde{S} T^{3+4q}\\&+&f_{-4}\vert_{2\beta}\tilde{S} T^{4+4q}\\
{\bf{L}}_{\beta,\chi_\alpha} \tilde{f}_{-1}&=&\sum\limits_{q=0}^\infty f_{+5}\vert_{2\beta} \tilde{S} T^{1+4q}+f_{+4}\vert_{2\beta}\tilde{S} 
T^{2+4q}+f_{+3}\vert_{2\beta}\tilde{S} T^{3+4q}\\&+&f_{+2}\vert_{2\beta}\tilde{S} T^{4+4q}\\
{\bf{L}}_{\beta,\chi_\alpha} \tilde{f}_{-2}&=&\sum\limits_{q=0}^\infty e^{-2 \pi i \alpha(1+4q)}f_{+1}\vert_{2\beta} \tilde{S} T^{1+4q}+e^{-2 \pi i \alpha(2+4q)}f_{+1}\vert_{2\beta}\tilde{S} 
T^{2+4q}\\&+&e^{-2 \pi i \alpha(3+4q)}f_{+1}\vert_{2\beta}\tilde{S} T^{3+4q}+e^{-2 \pi i \alpha(4+4q)}f_{+1}\vert_{2\beta}\tilde{S} T^{4+4q}\\
{\bf{L}}_{\beta,\chi_\alpha} \tilde{f}_{-3}&=&\sum\limits_{q=0}^\infty e^{-2\pi i \alpha}f_{+4}\vert_{2\beta} \tilde{S} T^{1+4q}+e^{-2\pi i \alpha}f_{+3}\vert_{2\beta}\tilde{S} 
T^{2+4q}\\&+&e^{-2\pi i \alpha}f_{+2}\vert_{2\beta}\tilde{S} T^{3+4q}+e^{-2\pi i \alpha}f_{+5}\vert_{2\beta}\tilde{S} T^{4+4q}\\
{\bf{L}}_{\beta,\chi_\alpha} \tilde{f}_{-4}&=&\sum\limits_{q=0}^\infty e^{2 \pi i \alpha (1+4q)}f_{+6}\vert_{2\beta} \tilde{S} T^{1+4q}+e^{2 \pi i \alpha (2+4q)}f_{+6}\vert_{2\beta}\tilde{S} 
T^{2+4q}\\&+&e^{2 \pi i \alpha (3+4q)}f_{+6}\vert_{2\beta}\tilde{S} T^{3+4q}+e^{2 \pi i \alpha (4+4q)}f_{+6}\vert_{2\beta}\tilde{S} T^{4+4q}\end{eqnarray*}\begin{eqnarray*}
{\bf{L}}_{\beta,\chi_\alpha} \tilde{f}_{-5}&=&\sum\limits_{q=0}^\infty e^{2\pi i \alpha}f_{+2}\vert_{2\beta} \tilde{S} T^{1+4q}+e^{2\pi i \alpha}f_{+5}\vert_{2\beta}\tilde{S} 
T^{2+4q}\\&+&e^{2\pi i \alpha}f_{+4}\vert_{2\beta}\tilde{S} T^{3+4q}+e^{2\pi i \alpha}f_{+3}\vert_{2\beta}\tilde{S} T^{4+4q}\\
{\bf{L}}_{\beta,\chi_\alpha} \tilde{f}_{-6}&=&\sum\limits_{q=0}^\infty f_{+3}\vert_{2\beta} \tilde{S} T^{1+4q}+f_{+2}\vert_{2\beta}\tilde{S} T^{2+4q}+f_{+5}\vert_{2\beta}\tilde{S} T^{3+4q}\\&+&f_{+4}\vert_{2\beta}\tilde{S} T^{4+4q}
\end{eqnarray*} 
\label{C1}
where $\tilde{f}\in B(D,\mathbb{C}^{\mu})\bigoplus B(D,\mathbb{C}^{\mu})$ is given by $\tilde{f}=(\underline{f}_+,
\underline{f}_-)$ and $\underline{f}_{\pm}=(f_{\pm i})_{1\leq i\leq 6}$ and $\tilde{S}z=\frac{1}{z}$. The induced representation $\rho_\chi$ of the 
 character $\chi$ on $\Gamma_0(4)$ is defined  in terms of the coset decomposition of $SL(2,\mathbb{Z})$
\begin{equation}\label{coset}
SL(2,\mathbb{Z})=\bigcup\limits_{i=1}^6 \Gamma_0(4) R_i 
\end{equation}
 as 
\begin{equation}\label{rep}
\rho_\chi(g)_{i\;j}= 
\delta_{\Gamma_0(4)}(R_igR_j^{-1}) \chi(R_igR_j^{-1}),\quad 1\leq i,j\leq 6.
\end{equation}
Thereby we have chosen the  following representatives $R_i\in SL(2,\mathbb{Z})$ of the cosets $\Gamma_0(4)\,R_i$  
\begin{equation}\label{Repres}
R_1=id_2 ,\; R_i=ST^{i-2}, 2\leq i\leq 5 
\quad\text{and} \quad R_6= ST^2S.
\end{equation} 
 It turns out that the two  permutation matrices $P_i, i=1,2$ corresponding to the permutations $\sigma _i$ with
 \begin{equation}\label{sym1}	
\sigma_1=
	\begin{tabular}{cccccc}
	 1 & 2&3&4&5&6 \\
	 \hline 
	 1 & 2&5&4&3&6
	\end{tabular}  
		\end{equation}
\begin{equation}\label{sym2}		
	\sigma_2=
	\begin{tabular}{cccccc}
	 1 & 2&3&4&5&6 \\
	 \hline 
	 6 & 4&3&2&5&1 
	\end{tabular}  
	\end{equation}
fulfill equation (\ref{symmetry}) for $\alpha =0$ and hence the corresponding matrices $\tilde{P}_i,\, i=1,2$ commute with the transfer operator ${\bf{L}}_{\beta,\chi_0}$ where $\chi_0$ is the trivial character. The matrix $\tilde{P}_2$ on the other hand commutes even with the operator ${\bf{L}}_{\beta,\chi_\alpha}$ for all $\alpha$. Indeed the matrix $\rho_{\chi_0}(S)$ is given by the permutation $\sigma_S$ where 
\begin{equation}\label{permS}	
\sigma_S=
	\begin{tabular}{cccccc}
	 1 & 2&3&4&5&6 \\
	 \hline 
	 2 & 1&5&6&3&4
	\end{tabular}  
		\end{equation} and an easy calculation shows that $P_i \rho_{\chi_0}(S)=\rho_{\chi_0}(S) P_i,\; i=1,\, 2$. The matrix $\rho_{\chi_0}(T)$ on the other hand is given by the permutation $\sigma_T$ with 
		\begin{equation}\label{permT}	
\sigma_T=
	\begin{tabular}{cccccc}
	 1 & 2&3&4&5&6 \\
	 \hline 
	 1 &3&4&5&2&6
	\end{tabular} . 
		\end{equation}
One then checks that $P_i \rho_{\chi_0}(T)=\rho_{\chi_0}(T^{-1}) P_i,\; i=1,\, 2$. Therefore $P_i \rho_{\chi_0}(ST^n)= \rho_{\chi_0}(ST^{-n})P_i$ for all $n\in \mathbb{N}$ and $i=1,\; 2$.	For the character $\chi_\alpha $ analogous relations hold for $P_2$.\newline	
For the trivial character $\chi_0$ one can determine for the group $\Gamma_0(N)$ the number $h_N$ of matrices $P_i$ with the above properties and hence defining symmetries of the transfer operator as follows:
\begin{theorem}\label{symmetry}
For the Hecke congruence subgroup $\Gamma_0(N)$ and trivial character $\chi_0\equiv 1$ there exist $h_N$ matrices 
$\tilde{P}=\begin{pmatrix}0&P\\P&0\end{pmatrix}$  commuting with the transfer operator ${\bf{L}}_{\beta,\chi_0}$ where $P$ is a 
$\mu_N\times\mu_N$ permutation matrix with $P^2=\id_{\mu_N}$ and $P \rho_{\chi_0}(S)= \rho_{\chi_0}(S) P$ respectively $P \rho_{\chi_0}(T)= \rho_{\chi_0}(T^{-1}) P$ and hence
\begin{equation*}
P\,\mathcal{L}_{\beta,\chi_0}^+ =\mathcal{L}_{\beta,\chi_0}^-P.
\end{equation*} 
Thereby  $h_N= max\{k: k\mid 24\quad\text{and}\quad k^2\mid N\}$.  
The permutation matrices $P$ are determined by the $h_N$ generators $j$ of the normalizer group $\mathcal{N}_N$  of $\Gamma_0(N)$ in $GL(2,\mathbb{Z})$. The Selberg zeta function $Z_{\Gamma,\chi_0}$ can be written as 
$$Z_{\Gamma,\chi_0}= \det (1-P\,\mathcal{L}_{\beta,\chi_0}^+)\det (1+P\,\mathcal{L}_{\beta,\chi_0}^+).$$

\end{theorem}.
\begin{remark}
For $\Gamma_0(4)$  obviously $h_N=2$  and there exist according to Theorem \ref{symmetry} two such permutation matrices $P_1$ and $P_2$ which  indeed are given by the aforementioned permutations $\sigma_i,\, i=1,2$.  Since $P_1\,P_2 = P_2\,P_1$ and $P_i\,\mathcal{L}_{\beta,\chi_0}^+=\mathcal{L}_{\beta,\chi_0}^-\, P_i,\, i=1,2$ we find \begin{center}$P_1\,P_2\, P_1\,\mathcal{L}_{\beta,\chi_0}^+=P_1\,P_2\, 
\mathcal{L}_{\beta,\chi_0}^- \,P_1=P_1\,\mathcal{L}_{\beta,\chi_0}^+P_2\,P_1=P_1\,\mathcal{L}_{\beta,\chi_0}^+P_1\,P_2$ \end{center}
and the operators $P_1\,P_2$ and $P_1\,\mathcal{L}_{\beta,\chi_0}^+$ commute, where the operator $P_1\,P_2$ corresponds to the permutation 		\begin{equation}\label{perm sigma}	
\sigma=
	\begin{tabular}{cccccc}
	 1 & 2&3&4&5&6 \\
	 \hline 
	 6 &4&5&2&3&1
	\end{tabular} . 
		\end{equation}
 We find also $P_1\,P_2\,\mathcal{L}_{\beta,\chi_0}^+=\mathcal{L}_{\beta,\chi_0}^+\,P_1\,P_2$. But $(P_1\,P_2)^2=id_{6}$, hence this 
operator has only the eigenvalues $\lambda=\pm 1$ and the Banach space $B(D,\mathbb{C}^{6})$ decomposes as 
$B(D,\mathbb{C}^{6})=B(D,\mathbb{C}^{6})_+ \oplus B(D,\mathbb{C}^{6})_-$ with 
$P_1\,P_2\underline{f}_{\pm}=\pm \underline{f}_\pm$ for $\underline{f}_\pm \in B(D,\mathbb{C}^{6})_\pm$. The elements $\underline{f}_\epsilon\in B(D,\mathbb{C}^{6})_\epsilon,\, \epsilon=\pm $ have therefore the form $(\underline{f}_\epsilon)_i=f_i,\, 1\leq i\leq 3$ respectively $(\underline{f}_\epsilon)_{\sigma(i)}=\epsilon f_i,\, 1\leq i\leq 3$.
Denote by \begin{center}
$\mathcal{L}_{\beta,\chi_0,\pm}^+:B(D,\mathbb{C}^{6})_\pm \to B(D,\mathbb{C}^{6})_\pm$
\end{center}
respectively
\begin{center}$P_1\,\mathcal{L}_{\beta,\chi_0,\pm}^+:B(D,\mathbb{C}^{6})_\pm \to B(D,\mathbb{C}^{6})_\pm$
\end{center} 
the restriction of the operators $ \mathcal{L}_{\beta,\chi_0}^+$ respectively $P_1\,\mathcal{L}_{\beta,\chi_0}^+$ to the subspace $B(D,\mathbb{C}^{6})_\pm$, which obiously is isomorphic to the space $B(D,\mathbb{C}^{3})$. Then 
$\det (1\pm P_1\,\mathcal{L}_{\beta,\chi_0}^+)=\det (1\pm P_1\,\mathcal{L}_{\beta,\chi_0,+}^+)\det (1\pm P_1\,\mathcal{L}_{\beta,\chi_0,-}^+)$, where the operator $ P_1\,\mathcal{L}_{\beta,\chi_0,\epsilon}^+:B(D,\mathbb{C}^{3})\to B(D,\mathbb{C}^{3})$ can be written as 
\begin{equation}
 P_1\,\mathcal{L}_{\beta,\chi_0,\epsilon}^+=\begin{pmatrix} 0&\epsilon\mathcal{L}_{\beta,2}+\mathcal{L}_{\beta,4}&\epsilon \mathcal{L}_{\beta,1}+\mathcal{L}_{\beta,3}\\\mathcal{L}_\beta&0&0\\0& \mathcal{L}_{\beta,1}+ \epsilon\mathcal{L}_{\beta,3}&\epsilon 
\mathcal{L}_{\beta,2}+\mathcal{L}_{\beta,4}
\end{pmatrix}. 
\end{equation}
with $\mathcal{L}_{\beta,k}f=\sum\limits_{q=0}^\infty f|_{2\beta}\tilde{S}T^{1+k q},\, 1\leq k\leq 4$ and $\mathcal{L}_\beta=\sum\limits_{k=1}^4 \mathcal{L}_{\beta,k}$. The operator $\mathcal{L}_{\beta,\chi_0,\epsilon}^+$ in the space $B(D,\mathbb{C}^{3})$ on the other hand has the form 

\begin{equation}
 \mathcal{L}_{\beta,\chi_0}^+{_\epsilon}=\begin{pmatrix} 0&\epsilon\mathcal{L}_{\beta,2}+\mathcal{L}_{\beta,4}&\epsilon \mathcal{L}_{\beta,1}+\mathcal{L}_{\beta,3}\\\mathcal{L}_\beta&0&0\\0& \epsilon \mathcal{L}_{\beta,1}+ \mathcal{L}_{\beta,3}&\mathcal{L}_{\beta,2}+\epsilon\mathcal{L}_{\beta,4}
\end{pmatrix}.
\end{equation}
To relate the Fredholm determinants of the operators $(P_1\mathcal{L}_{\beta,\chi_0,\epsilon}^+)^2 $ and $(\mathcal{L}_{\beta,\chi_0,\epsilon}^+)^2$ we use the following simple Lemma
\begin{lemma}
Let be $\alpha, \beta$ and $\gamma$ complex numbers and $\epsilon=\pm 1$. Then $\lambda$ is an eigenvalue of the matrix $\mathbb{L}_1=\begin{pmatrix}
0&\alpha&\beta \\ \gamma &0&0 \\ 0&\beta&\epsilon\alpha
\end{pmatrix}$ iff $\epsilon \lambda$ is an eigenvalue of the matrix  $\mathbb{L}_2=\begin{pmatrix}
0&\alpha&\beta \\ \gamma &0&0 \\ 0&\epsilon \beta&\alpha
\end{pmatrix}$. 
\end{lemma}
\begin{proof}
The proof follows from the characteristic polynomial of the two matrices. \end{proof}
This shows that $\text{trace}\, \mathbb{L}_1^{n}=\sum\limits_{k=1}^3 {(\mathbb{L}_1^{n})}_{k,k}= \epsilon^n \;\text{trace} \, \mathbb{L}_2^{n}=\epsilon^n \sum\limits_{k=1}^3 {(\mathbb{L}_2^{n})}_{k,k}$ for all $n\in \mathbb{N}$.
But then is not too difficult to see that also $ \text{trace} \,(\mathcal{L}_{\beta,\chi_0,\epsilon}^+)^{n}=\epsilon^n\, \text{trace} \,(P_1\mathcal{L}_{\beta,\chi_0,\epsilon}^+)^{n}$ for all $n\in\mathbb{N}$ and hence $\det (1- (P_1 \mathcal{L}_{\beta,\chi_0,\epsilon}^+)^2)=\det (1-(\mathcal{L}_{\beta,\chi_0,\epsilon}^+)^2)$ for $\epsilon =\pm$. Therefore the Selberg zeta function $Z_{\Gamma_0(4),\chi_0}(\beta)$ for the group $\Gamma_0(4)$ with trivial character $\chi_0$ can be written as 
\begin{eqnarray} Z_{\Gamma_0(4),\chi_0}(\beta)&=&\det \left(1- (P_1 \mathcal{L}_{\beta,\chi_0}^+)^2\right)=\det \left(1-(\mathcal{L}_{\beta,\chi_0}^+)^2\right)\nonumber\\
&=&\det (1-\mathcal{L}_{\beta,\chi_0}^+)\det(1+\mathcal{L}_{\beta,\chi_0}^+)
\end{eqnarray}
Furthermore this function factorizes in this case also as 
\begin{eqnarray}Z_{\Gamma_0(4),\chi_0}(\beta)&=& \det 
(1-P_1\,\mathcal{L}_{\beta,\chi_0,+}^+)\det (1-P_1\,\mathcal{L}_{\beta,\chi_0,-}^+)\nonumber \\&\times&\det (1+P_1\,\mathcal{L}_{\beta,
\chi_0,+}^+)\det (1+P_1\,\mathcal{L}_{\beta,\chi_0-}^+)\end{eqnarray}
\end{remark}
To prove  Theorem \ref{symmetry} we relate the matrices $P$ to the generating automorphisms in $GL(2,\mathbb{Z})$ of the Maass wave forms for $\Gamma_0(N)$ and can determine this way the explicit form of these matrices $P$. For this we  derive  in a first step a Lewis type functional equation for the eigenfunctions of the 
operator $P\,\mathcal{L}_{\beta,\chi}^+ $ with eigenvalue $\lambda = \pm 1$.  

\subsection{A Lewis type functional equation}

Consider any finite index modular subgroup $\Gamma$ and any unitary character $\chi:\Gamma\to \mathbb{C}^\star$ respectively the induced representation $\rho_\chi$ of $SL(2,\mathbb{Z})$.  
Assume there exists a symmetry $\tilde{P}=\begin{pmatrix}0&P\\P&0\end{pmatrix}$ with  $P$ a permutation matrix with the properties analogous to  Theorem \ref{symmetry}, and commuting with the transfer operator 
${\bf{L}}_{\beta,\chi}=\begin{pmatrix}0&\mathcal{L}_{\beta,\rho_\chi}^+\\ \mathcal{L}_{\beta,
\rho_\chi}^-&0\end{pmatrix}$ of $\Gamma$. If $\underline{f}$ is an eigenfunction of the operator $P\,\mathcal{L}_{\beta,\chi}^+$ with 
eigenvalue $\lambda = \pm 1$ then one can show 
\begin{proposition}
If $P\,\mathcal{L}_{\beta,\chi}^+\underline{f}(\zeta)=\lambda \underline{f}(\zeta)$ with $\lambda =\pm 1$ 
then the function $\underline{\Psi}(\zeta):= P \rho_\chi (T^{-1}S)P \underline{f} (\zeta-1) $ fulfills the functional equations
\begin{equation}\label{Lewis1}
\underline{\Psi}(\zeta)= 
\lambda \zeta^{-2 \beta} P \,\rho_\chi (S) \,\underline{\Psi}(\frac{1}{\zeta}),
\end{equation}
respectively
\begin{equation}\label{Lewis2}
\underline{\Psi}(\zeta) -\rho_\chi(T^{-1})\,\underline{\Psi}(\zeta +1) - (\zeta +1)^{-2\beta}\,\rho_\chi(T'^{-1})\,\underline{\Psi}(\frac{\zeta}{\zeta + 1}) =\underline{0}, 
\end{equation}
 where  $ T'= S T^{-1} S$.
 On the other hand every solution $	\underline{\Psi}$ of equations (\ref{Lewis1}) and (\ref{Lewis2}) holomorphic in the cut $\beta$-plane $(-\infty,0]$ with $\Psi_i(z)=o(z^{-\min\{1,2\Re s\}})$ as $z\downarrow 0$, respectively $\Psi_i(z)=o(z^{-\min\{0,2\Re s-1\}})$ as $z \to \infty$, determines an eigenfunction $\underline{f}$ with eigenvalue $\lambda =\pm 1$ of the operator $P\,\mathcal{L}_{\beta,\chi}^+$.
\end{proposition}
\begin{proof}
 Let $\Re \beta > \frac{1}{2}$. If	$P \mathcal{L}_\beta^+\underline{f}(\zeta)=\lambda \underline{f}(\zeta),\, \lambda=\pm 1$ 
	then obviously \newline
	$P \rho_\chi(STS) P P \mathcal{L}_\beta^+\underline{f}(\zeta+1)=\lambda P \rho_\chi(STS) P \underline{f}(\zeta+1).$
	Subtracting the two equations leads to
$$ \lambda \underline{f}(\zeta) - \lambda P \rho_\chi(STS) P \underline{f}(\zeta+1)-(\zeta+1)^{-2\beta} P \rho_\chi(ST)\underline{f}(\frac{1}{\zeta+1})= \underline{0},$$
	and hence the function $\underline{\psi}(\zeta):= P \underline{f}(\zeta-1)$ fulfills the equation 
	\begin{equation}\label{Eq1}
	\underline{\psi}(\zeta)-\rho_\chi(STS)\underline{\psi}(\zeta+1)- \lambda \zeta^{-2 \beta}\rho_\chi(ST) P\underline{\psi}(\frac{\zeta+1}{\zeta})= \underline{0}. 
	\end{equation}
	Replacing there $\zeta$ by $\frac{1}{\zeta}$ and multiplying the resulting equation by  $\zeta^{-2 \beta} \rho_\chi (STS) P \rho_\chi(T^{-1}S)$ gives
	$$\zeta^{-2\beta}\rho_\chi(STS) P \rho_\chi (T^{-1}S)\underline{\psi}(\frac{1}{\zeta})-\zeta^{-2\beta}\rho_\chi(STS) P \rho_\chi(S)\underline{\psi}(\frac{\zeta+1}{\zeta})-$$
	$$-\lambda  \rho_\chi(STS) \underline{\psi}(\zeta+1)= \underline{0} .  $$
	
	Since $\rho_\chi (S) P = P \rho_\chi(S)$ one finds, comparing with equation (\ref{Eq1}),
  $$\underline{\psi}(\zeta)=\lambda \zeta^{-2\beta}\rho_\chi(STS) P \rho_\chi(T^{-1}S)\underline{\psi}(\frac{1}{\zeta}).$$
   Hence the function
		 $\tilde{\underline{\psi}}:=\rho_\chi(T^{-1}S)\underline{\psi}$ fulfills  equation (\ref{Lewis1}).
	The same equation is then fulfilled also  by the function 	  
\begin{equation}\label{psif}\underline{\Psi}(\zeta):=  P\tilde{\underline{\psi}}(\zeta )=P \,\rho_\chi(T^{-1}S)\, P\, \underline{f}(\zeta-1),\end{equation} 
that is
\begin{equation}
\underline{\Psi}(\zeta )= \lambda \zeta^{-2\beta} P \rho_\chi (S)  \underline{\Psi}(\frac{1}{\zeta}) .
\end{equation}
Inserting finally $\underline{\psi}(\zeta)=\rho_\chi(ST)P \underline{\Psi}(\zeta)$ into equation (\ref{Eq1}) and using (\ref{Lewis1}) leads to the equation 
$$\underline{\Psi}(\zeta)-P \rho_\chi(T)P \underline{\Psi}(\zeta+1)-  (\zeta+1)^{-2 \beta} P\rho_\chi(T') P\underline{\Psi}(\frac{\zeta}{\zeta+1})= \underline{0}. $$
But by assumption $P \rho_\chi(T) P= \rho_\chi(T^{-1})$, hence $P\rho_\chi(T')P= \rho_\chi(T'^{-1})$ and therefore
\begin{equation}
\underline{\Psi}(\zeta)-\rho_\chi(T^{-1}) \underline{\Psi}(\zeta+1)-  (\zeta+1)^{-2 \beta} \rho_\chi(T'^{-1}) \underline{\Psi}(\frac{\zeta}{\zeta+1})= \underline{0}.
\end{equation}
Hence for $\Re \beta >\frac{1}{2}$ the first part of the proposition holds. By analytic continuation in $\beta$ one proves the general case.\\
To prove the second part we follow the arguments of Deitmar and Hilgert in \cite{DH07} (see their Lemma 4.1): if 
$\underline{\Psi}(\zeta)$ is a solution of the Lewis equation (\ref{Lewis2}) with $\beta\notin \mathbb{Z}$ then $\underline{\Psi}$ has the 
following asymptotic expansions:$$\underline{\Psi}(\zeta)\thicksim_{\zeta\to 0} \zeta^{2 \beta} Q_0(\frac{1}{\zeta})+\sum_{l=-1}^\infty \underline{C}_l^* \zeta^l,$$
$$\underline{\Psi} (\zeta)\thicksim_{\zeta\to\infty} Q_\infty (\zeta)+\sum_{l=-1}^\infty 
\underline{C}_l^{*'}
\zeta^{-l-2 \beta},
$$
where $Q_0,Q_\infty:\mathbb{C}\to \mathbb{C}^\mu$ are smooth functions with $Q_0(\zeta+1)=\rho_\chi(T') Q_0(\zeta)$ 
respectively $Q_\infty (\zeta+1)=\rho_\chi(T) Q_\infty (\zeta)$ and the constants $\underline{C}_l^*$ and $\underline{C}_l^{*\prime}$ are determined by the 
Taylor coefficients $\underline{C}_m=\frac{1}{m!}\underline{\Psi}^{(m)}(1)$. The functions $Q_0$ and $Q_\infty$ are defined as follows for 
general $\beta$ with $-2 \Re \beta <M\in\mathbb{N}$:
\begin{eqnarray*}
Q_0(\zeta)&:=&\zeta^{-2\beta} \underline{\Psi}(\frac{1}{\zeta})-\sum\limits_{m=0}^M \zeta_{\rho_\chi}
(m+2\beta,z)\underline{C}_m\\&-&\sum\limits_{n=0}^\infty 
(n+\zeta)^{-2\beta}\rho_\chi(T'^{-n}T^{-1})\left(\underline{\Psi}(1+\frac{1}{n+\zeta})
-\sum\limits_{m=0}^M \frac{\underline{C}_m}
{(n+\zeta)^m)}\right)
\end{eqnarray*} 
respectively
\begin{eqnarray*}
Q_\infty(\zeta)&:=& \underline{\Psi} (\zeta)-\sum\limits_{m=0}^M 
\zeta_{\rho_\chi}^{'} (m+2\beta,\zeta+1) \underline{C}_m\\&-&
\sum\limits_{n=0}^\infty 
(n+\zeta)^{-2\beta}\rho_\chi(T^{-(n-1)}T'^{-1})\left(\underline{\Psi}(1-\frac{1}{n+\zeta})-\sum\limits_{m=0}^M \frac{\underline{C}_m}{(n+\zeta)^m)}\right).
\end{eqnarray*}
Thereby $$\zeta_{\rho_\chi}(a,\zeta)=\frac{1}{N^a}\sum\limits_{k=0}^{N-1}\rho_\chi(T'^{-k}T^{-1})\zeta(a,\frac{k+\zeta}{N})$$ 
and $$\zeta'_{\rho_\chi}(a,\zeta)=\frac{1}{N^a}\sum\limits_{k=0}^{N-1}\rho_\chi(T^{-k}T'^{-1})\zeta_H(a,\frac{k+\zeta}{N})$$
with $\zeta_H(a,\zeta)$ the Hurwitz zeta function. According to Remark 4.2 in (\cite{DH07}) any solution $\underline{\Psi}$ of equation (\ref{Lewis2}) with $\underline{\Psi}(\zeta) =\underline{o}(\zeta^{-\min\{1,2\beta\}})$ for $\zeta\to 0$ fulfills the equation $$\underline{\Psi}(\zeta)=\zeta^{-2\beta} \sum_{n=0}^\infty (n+\zeta^{-1})^{-2\beta}\rho_\chi(T'^{-n}T^{-1})\underline{\Psi}(1+\frac{1}{n+\zeta^{-1}})$$
 and moreover $\underline{C}_{-1}^{*}=0$. But if $\underline{\Psi}(\zeta)$ fulfills also the equation (\ref{Lewis1}) then one finds
 $$\lambda \zeta^{-2\beta}P \rho_\chi(S) \underline{\Psi}(\frac{1}{\zeta})=\zeta^{-2\beta}\sum\limits_{n=0}^\infty (n+\zeta^{-1})^{-2\beta}\rho_\chi(T'^{-n}T^{-1})\underline{\Psi}(1+\frac{1}{n+\zeta^{-1}})$$ and hence
\begin{equation}\label{Gl}\lambda P \rho_\chi(S) \underline{\Psi}(\zeta+1)=\sum\limits_{n=1}^\infty (n+\zeta)^{-2\beta}\rho_\chi(T'^{-(n-1)}T^{-1})\underline{\Psi}(1+\frac{1}{n+\zeta}).\end{equation}
According to equation (\ref{psif}) $\underline{\Psi}(\zeta+1)=P\rho_\chi (T^{-1}S)P\underline{f}(\zeta)$ and hence  we get
$$\lambda \rho_\chi(ST^{-1}S) P \underline{f}(\zeta)=\sum\limits_{n=1}^\infty (n+\zeta)^{-2\beta}\rho_\chi(T'^{-(n-1)}T^{-1})P \rho_\chi(T^{-1}S)P\underline{f}(\frac{1}{\zeta+n}).$$
Inserting $T'^{-(n-1)}=ST^{(n-1)}S$ one arrives at
$$\lambda  \underline{f}(\zeta)=\sum\limits_{n=1}^\infty (n+\zeta)^{-2\beta} P \rho_\chi(S T^n)\rho_\chi(S T^{-1})P \rho_\chi(T^{-1} S) P \underline{f}(\frac{1}{\zeta+n}).$$
Since $\rho_\chi(ST^{-1})P=P\rho_\chi(ST)$ we get finally 
$$\lambda \underline{f}(\zeta)= \sum\limits_{n=1}^\infty \frac{1}{(n+\zeta)^{2\beta}}P \rho_\chi(ST^n)\underline{f}(\frac{1}{n+\zeta}).$$
Hence any solution $\underline{\Psi}$ of the Lewis equations (\ref{Lewis1}) and (\ref{Lewis2}) with the asymptotics at the cut $\zeta=0$ determines an eigenfunction $\underline{f}$ of the transfer operator $P \mathcal{L}_{\beta,\chi}^+$ with eigenvalue $\lambda=\pm 1$.

\end{proof}

\section{Automorphism  of  the Maass forms and their period functions for $\Gamma_0(N)$}
The Maassforms $u=u(z)$ of a cofinite Fuchsian group $\Gamma$ and unitary character $\chi$ are  real analytic functions $u:\mathbb{H}\to\mathbb{C}$ with
\begin{itemize}
\item $\Delta\, u(z)=\lambda \,u(z)$,
\item $u(gz)=\chi (g) \,u(z)$ for all $g \in \Gamma$,
\item $u(g_jz)=O(y^C$ as $y \to \infty$ for some constant $C\in \mathbb{R}$ and  all cusps $z_j=g_j(i\infty)$ of $\Gamma$.
\end{itemize}
The cusp forms are those forms which decay exponentially fast at the cusps. If $u\in L_2(M_\Gamma)$ we call $u$ a Maass wave form.
\begin{definition}
An element $j\in GL(2,\mathbb{Z})$ defines an automorphism $J$ of the Maass wave form $u$ for the group $\Gamma$ and character $\chi$ if  $J\,u$ with $Ju(z):=u(jz)$ is a Maass form for $\Gamma$ and character $\chi$.  
\end{definition}
Obviously $j$ defines an automorphism $J$ iff $j$ is a normalizer of the group $\Gamma$ and the character $\chi$ is invariant under $j$, that 
is $\chi (j \,g \,j^{-1}) =\chi (g) $  for all $g\in \Gamma$. Thereby $j z =\frac{az^*+b}{cz^*+d}$ if $\det g = ad-bd = -1$. We have to show  
that the function $Ju(z)=u(jz)$ has at most polynomial growth at the cusps $z_i=\tau_i(i \infty)$ of $\Gamma$, where  $\tau_i\in SL(2,\mathbb{Z)}$ . If $\det j=-1$, one has $u(j\tau_i(z))=u(j\tau_ij_{0,-}j_{0,-}(z))$ where $j_{0,-}=\begin{pmatrix}1&0\\0&-1\end{pmatrix}$. Then 
$j\tau_ij_{0,-}\in SL(2,\mathbb{Z})$ and hence $j\tau_ij_{0,-}=\gamma_i R_i$ for some $\gamma_i\in\Gamma$ and some representative $R_i$ 
of the cosets $\Gamma\setminus SL(2,\mathbb{Z})$. But $R_i=\eta\tau_{\sigma(i)}$ for some  $\eta\in\Gamma$ and some index $\sigma(i)$.  
Hence $u(j\tau_i(z))=u\left(\tau_{\sigma(i)}(-z^*)\right)$ which is at most of polynomial growth at the cusps. The same argument applies if $\det j =1$. It shows also that $Ju$ is a Maass wave form or a cusp form if $u$ is one.

\subsection{The group of automorphisms of Maass forms for $\Gamma_0(N)$ and trivial character $\chi_0$}

We restrict ourselves now to the case $ \Gamma =\Gamma_0(N)$ and assume $\chi =\chi_0$. Denote  by $\mathcal{N}_N$ the normalizer group $\{\Gamma_0(N) \, j: j \; \text{normalizer of} \; \Gamma_0(N) \; \text{in} \; GL(2,\mathbb{Z})\}$. Using results by Lehner and Newman \cite{LN64} respectively Conway and Norton \cite{CN79} we find
\begin{proposition}
For $h_N=\max \{r: r\mid 24 \,\text{and} \; r^2\mid N\}$ and $k_N:=\frac{N}{h_N}$ the normalizer group $\mathcal{N}_N$is given by  $\mathcal{N}_N=\{\Gamma_0(N)\,j_{n,\pm},\, j_{n,\pm}=\begin{pmatrix} 1&0\\n k_N&\pm 1\end{pmatrix}, 0\leq n\leq h_{N}-1\}$
\end{proposition}
\begin{proof}
Using the fact that the divisors $k$ of $24$ are exactly the numbers for which $a \cdot d=1\mod k$ implies $a=d \mod k$ one shows that the 
normalizer group of $\Gamma_0(N)$ in $ SL(2,\mathbb{Z})$ is $\Gamma_0(N)\setminus \Gamma_0(\frac{N}{\nu})$ \cite{LN64} with 
$\nu=2^{\min\{3,[\frac{\epsilon_2}{2}]\}}\cdot 3^{\min\{1,[\frac{\epsilon_3}{2}]\}}$ and $\epsilon_2=\max\{l: 2^l\mid N\}$ respectively 
$\epsilon_3=\max\{l: 3^l\mid N\}$. But obviously $\nu =h_N$ and $[\Gamma_0(k_N):\Gamma_0(N)]=h_N$ and hence 
$\mathcal{N}_N=\Gamma_0(N)\setminus \left(\Gamma_0(k_N)\bigcup \Gamma_0(k_N) j_{0,-}\right)$. Since $ j_{n,\pm} \not= j_{m,\pm}\mod \Gamma_0(N) $ for 
$n\not=m$, this group has just the $2 h_N$ elements $\Gamma_0(N) j_{n,\pm},\, 0\leq n \leq h_N-1$. The normalizer group $\mathcal{N}_N$
 is therefore generated by the $h_N$ generators $\{\Gamma_0(N)  j_{n,-},\,0\leq n\leq h_N-1\}$.
\end{proof}

\subsection{ The period functions of $\Gamma_0(N)$ and character $\chi$}
For $u$ a Maass form with  $\Delta u =\beta(1-\beta) u$ and $\Gamma_0(N)\setminus SL(2,\mathbb{Z})=\{\Gamma_0(N) R_i,\, 1\leq i\leq \mu_N\}$ its vector valued period function $\underline{u}$ is defined by
\begin{equation}\label{period}
\underline{u}=\left(u_i(z)\right)_{1\leq i\leq \mu_N} \;\text{where} \; u_i(z)=u(R_i z)
\end{equation}
Then one has as shown for instance in \cite{Mu06}:
\begin{itemize}
\item $\underline{u}(gz)=\rho_\chi(g) \underline{u}(z)$ for all $g\in SL(2,\mathbb{Z})$ and $\rho_\chi$ the  representation of $SL(2,\mathbb{Z})$ induced from the character $\chi$ on $\Gamma_0(N)$
\item $\Delta u_i(z)=\beta (1-\beta) u_i(z),\; 1\leq i\leq \mu_N$.
\end{itemize}
Given next two eigenfunctions $u=u(z)$ and $v=v(z)$ of the hyperbolic Laplacian with identical eigenvalue $\lambda=\beta(1-\beta)$,  one knows \cite{LZ01} that the $1$- form $\eta = \eta (u,v)$ with
$$\eta(u,v)(z):= {v(z) \partial_y u(z)-u(z)\partial_y v(z)}dx+[u(z)\partial_xv(z)-v(z)\partial_xu(z)]dy
$$
is closed. If $u=u(z)$ is a Maass wave form for $\Gamma_0(N)$ with eigenvalue $\lambda=\beta(1-\beta)$ and $R_\zeta(z)=\frac{y}{((\zeta-x)^2+y^2)^2}$ denotes the Poisson kernel, the vector valued period function $
\underline{\psi}=(\psi_j(\zeta))_{1\leq j\leq \mu_N}$ is defined as
\begin{equation}\psi_j(\zeta):= \int\limits_0^\infty \eta(u_j,R_\zeta^\beta)(z).\end{equation}
The following result has been shown for trivial character $\chi_0$ by M\"uhlenbruch in \cite{Mu06}. His proof can be extended however immediately to the case of a nontrivial character $\chi$.
\begin{proposition}
The period function $\underline{\psi}=\underline{\psi}(\zeta)$ of a Maass wave form $u=u(z)$ for $\Gamma_0(N)$ and unitary character $\chi$ is holomorphic in the cut $\zeta$-plane $\mathbb{C}\setminus (-\infty,0]$ and fullfills there the Lewis functional equation (\ref{Lewis2})
$$ \underline{\psi}(\zeta)-\rho_\chi(T^{-1})\underline{\psi}(\zeta +1)- (\zeta+1)^{-2\beta} \rho_\chi(T'^{-1})   \underline{\psi}(\frac{\zeta}{\zeta +1})=\underline{0},
$$ 
where $\rho_\chi$ denotes the representation of $SL(2,\mathbb{Z})$ induced from the character $\chi$ of $\Gamma_0(N)$.
\end{proposition}
On the other hand it follows from the work of Deitmar and Hilgert in \cite{DH07} that the solutions of the above equation 
holomorphic in the cut $\zeta$-plane with certain asymptotic behaviour at the cut $0$ and at $\infty$ are in one-to-one 
correspondence with the Maass wave forms. Their paper treats only the trivial character but it can be extended  also to the 
case of nontrivial character $\chi$. Since the function $\underline{\Psi}(\zeta)=P\rho_\chi(T^{-1}S)P \underline{f}
(\zeta-1)$ with $\underline{f}$ an eigenfunction of the operator $P \mathcal{L}_{\beta,\chi}^+$ with eigenvalue 
$\lambda=\pm 1$ is such a solution of equation (\ref{Lewis2}),  these eigenfunctions are in one-to-one correspondence 
with the Maass wave forms.  As in the case of the full modular group $SL(2,\mathbb{Z})$ treated in \cite{CM98} 
respectively in \cite{LZ01} one can extend this result to arbitrary Maass forms, that is also to the real analytic Eisenstein 
series for $\Gamma_0(N)$ and unitary character $\chi$
\subsection{Automorphisms of the period functions}
We have seen that the group of automorphisms in $GL(2,\mathbb{Z})$ of the  Maass forms $u$ of $\Gamma_0(N)$ and 
trivial character $\chi_0$ is generated by the matrices $j_{n,-}= \begin{pmatrix}1&0\\n k_N&-1\end{pmatrix},\, 0 \leq 
n\leq \mu_N-1$. Denote by $J_{n,-}u$ the Maass form $J_{n,-}u(z):=u(j_{n,-}z)$ and by $J_{n,-}\underline{\psi}$ its period 
function. Then one shows
\begin{theorem}
The period function $J_{n,-}\underline{\psi}=(J_{n,-}\psi_j(\zeta))_{1\leq j\leq \mu_N}$ is given by
\begin{equation}
J_{n,-}\psi_j(\zeta)=\zeta^{-2\beta}\psi_{\lambda_{n_-}\circ\sigma\circ\delta(j)}(\frac{1}{\zeta}),
\end{equation}
where the permutations $\lambda_{n_-},\sigma\;\text{and}\;\delta$ are determined through  the 
coset representatives $R_j$ of  $\Gamma_0(N)\setminus SL(2,\mathbb{Z})$ as follows: $$j_{n,+}R_j=\theta_j R_{\lambda_{n,-}(j)},\; j_{0,-}R_j j_{0,-}=\gamma_j R_{\sigma(j)} \;\text{and}\; R_j S= \eta_j R_{\delta(j)}$$ with $\theta_j, \gamma_j, \eta_j\in \Gamma_0(N)$ for $1\leq j\leq \mu_N$
\end{theorem}
\begin{proof}
For $u=u(z)$ a Maass form for $\Gamma_0(N)$ and trivial character $\chi_0$ and $\underline{u}=\underline{u}(z)$ its vector valued Maass form consider the Maass forms $J_{n,\pm}u(z)=u(j_{n,\pm}z)$ respectively  $J_{n,\pm}\underline{u}(z)=(J_{n,\pm}u_j(z))_{1\leq j\leq \mu_N}$ with $J_{n,\pm}u_j(z)=u(j_{n,\pm}R_jz)$. Since $j_{n,+}R_j=\theta_j R_{\lambda_{n.-}(j)}$ for some uniquely defined $\theta_j\in \Gamma_0(N)$ and permutation $\lambda_{n,-}$ of $\{1,2,\ldots,\mu_N\}$ one gets for $J_{n,+}u_j$
\begin{equation}\label{Eq2}
J_{n,+}u_j(z) = u(R_{\lambda_{n,-}(j)}z)=u_{\lambda_{n,-}(j)}(z).
\end{equation}
For $J_{n,+}u_j(-z^*)=u(j_{n,+}R_j(-z^*)=u(j_{n,+}R_j j_{0,-}z)$ on the other hand one finds $$J_{n,+}u_j(-z^*)=u(j_{n,-}j_{0,-}R_j j_{0,-}z)=u(j_{n,-}R_{\sigma(j)}z)$$ since $j_{0,-}R_j j_{0,-}= \gamma_j R_{\sigma(j)}$ for some unique $\gamma_j\in \Gamma_0(N)$ and permutation $\sigma$ of $\{1,2,\ldots,\mu_N\}$. Hence
\begin{equation}\label{Eq3}
J_{n,+}u_j(-z^*)=J_{n,-}u_{\sigma(j)}(z).
\end{equation}
Consider next $J_{n,+}u_j(Sz)= J_{n,+}u(R_jSz)$. Since $R_jS=\eta_j R_{\delta(j)}$ for  unique $\eta_j\in \Gamma_0(N)$ and permutation $\delta $ of $\{1,2,\ldots,\mu_N\}$, one has $$J_{n,+}u_j(Sz)=J_{n,+}u(R_{\delta(j)}z)= J_{n,+}u_{\delta(j)}(z).$$ Hence by equation (\ref{Eq2})
\begin{equation}\label{Eq4}
J_{n,+}u_j(Sz)=u_{\lambda_{n,-}\circ \delta(j)}(z).
\end{equation}
On the other hand one gets  for $J_{n,+}u_j(S(-z^*))=J_{n,+}u_j(-Sz^*)$ by using equation (\ref{Eq3}): 
$$J_{n,+}u_j(S(-z^*))= J_{n,-}u_{\sigma(j)}(Sz)= u(j_{n,-}R_{\sigma(j)}Sz),$$ and therefore
$$J_{n,+}u_j(-
Sz^*)=u(j_{n,-}\eta_{\sigma(j)}R_{\delta\circ\sigma(j)}(z))=J_{n,-}u_{\delta\circ\sigma(j)}(z).$$ 
But $\sigma\circ\delta=\delta\circ \sigma$ and  therefore
\begin{equation}\label{Eq5}
J_{n,+}u_j(S(-z^*))=J_{n,-}u_{\sigma\circ\delta(j)}(z).
\end{equation}
Define next $$v_{\pm,j}(z):= J_{n,+}u_j(z)\pm J_{n,+}u_j(-z^*).$$ Then by equations (\ref{Eq2}) and (\ref{Eq3}) one has
$$v_{\pm,j}(z)=u_{\lambda(j)}(z)\pm J_{n,-}u_{\sigma(j)}(z)$$ and hence, if $\Delta u(z)=\beta (1-\beta) u(z)$, 
\begin{equation}
\Delta v_{\pm,j}(z)=\beta (1-\beta) v_{\pm,j}(z) ,
\end{equation}
 respectively
\begin{equation}
v_{\pm,j}(-z^*)= \pm v_{\pm,j}(z)
\end{equation}

Equations (\ref{Eq4}) and (\ref{Eq5}) on the other hand show
\begin{equation}\label{v}
v_{\pm,j}(Sz)= v_{\pm,\delta(j)}(z).
\end{equation}
Set $\underline{\psi}'_\pm (\zeta):= \int\limits_0^{i\infty}\eta(\underline{v}_\pm,R_\zeta^\beta)(z)$. Then, since $v_{\pm,j}(-z^*)= \pm v_{\pm,j}(z)$ one finds \cite{LZ01}
\begin{equation}
\psi'_{+,j}(\zeta)=2\beta\int\limits_0^\infty \frac{t^\beta v_{+,j}(it)}{(\zeta^2+t^2)^{\beta+1}} \,dt,
\end{equation}
respectively
\begin{equation}
\psi'_{-,j}(\zeta)=-\int\limits_0^\infty \frac{t^\beta \partial_    x  v_{-,j}(it)}{(\zeta^2+t^2)^{\beta}} \,dt.
\end{equation}
Using next the identity (\ref{v}) one easily shows
\begin{equation}
\psi'_{\pm,j}(\zeta)=\pm \zeta^{-2 \beta} \psi'_{\pm,\delta(j)}(\frac{1}{\zeta}).
\end{equation}
But  $v_{\pm,j}(z)=u_{\lambda(j)}(z)\pm J_{n,-}u_{\sigma(j)}(z)$ and hence 
\begin{equation*}
\psi'_{\pm,j}(\zeta)=\psi_{\lambda_{n,-}(j)}(\zeta)\pm J_{n,-}\psi_{\sigma(j)}(\zeta).
\end{equation*}
Therefore 
\begin{equation}
\psi_{\lambda_{n,-}(j)}(\zeta)\pm J_{n,-}\psi_{\sigma(j)}(\zeta)=\pm \zeta^{-2 \beta}\left( \psi_{\lambda_{n,-}\circ\delta(j)}(\frac{1}{\zeta})\pm J_{n,-}\psi_{\sigma\circ\delta(j)}(\frac{1}{\zeta})\right)
\end{equation}

Adding these two equations leads finally to
\begin{equation}
\psi_{\lambda_{n,-}(j)}(\zeta)=\zeta^{-2\beta} J_{n,-}\psi_{\sigma\circ\delta(j)}(\frac{1}{\zeta}),
\end{equation}
and therefore to the equation
\begin{equation}
J_{n,-}\psi_j(\zeta)=\zeta^{-2\beta} \psi_{\lambda_{n,-}\circ\sigma\circ\delta(j)}(\frac{1}{\zeta}),
\end{equation}
which was to be proven.
\end{proof}
\begin{remark}
As can be seen from their action on the coset representatives $R_j$ the permutation $\delta$ commutes with the permutations 
$\lambda_{n,-}$ and $\sigma$. Furthermore one has $\sigma^2=\delta^2=(\lambda_{n,-}\circ\sigma)^2=id$ where $id$ 
denotes the identy permutation. This shows also that  the automorphisms $J_{n,-}$ are involutions both of the Maass 
forms and the period functions, a special case of these involutions for all groups $\Gamma_0(N)$ being $J_{0,-}u(z)=u(-
z^*)$.
\end{remark}
Denote by $Q_{n,-}, \, 0\leq n\leq h_N-1,$ the $\mu_N\times\mu_N$ permutation matrix corresponding to the permutation 
$\lambda_{n,-}\circ\sigma\circ\delta$. Then the following Theorem holds:
\begin{theorem}\label{Symm}
The permutation matrices $P_{n,-}:= \rho_{\chi_0}(S) Q_{n,-},\, 0\leq n\leq h_N-1,$ define  symmetries 
$\tilde{P}_{n,-}=\begin{pmatrix}0&P_{n,-}\\P_{n,-}&0\end{pmatrix}$ for the transfer operator ${\bf{L}}_{\beta,
\chi_0}=\begin{pmatrix}0&\mathcal{L}_{\beta,\chi_0}^+\\\mathcal{L}_{\beta,\chi_0}^+&0\end{pmatrix}$ for 
$\Gamma_0(N)$ and trivial character $\chi_0\equiv 1$ with $P_{n,-}^2= id_{\mu_N}$ and $P_{n,-}\rho_{\chi_0}(S)=\rho_{\chi_0}(S)P_{n,-}$ respectively $P_{n,-}\rho_{\chi_0}(T)=\rho_{\chi_0}(T^{-1})P_{n,-}$ and therefore $P_{n,-}\mathcal{L}_{\beta,\chi_0}^+=\mathcal{L}_{\beta,\chi_0}^+ P_{n,-}$. The permutation matrix $P_{n,-}$ is determined by the permutation 
$\lambda_{n,-}\circ\sigma$ and hence by the coset representatives $j_{n,-}R_j j_{0,-}$.
\end{theorem}
\begin{proof}
Since the matrix $P_{n,-} \rho_{\chi_0}(S)$ is determined by the coset representatives $j_{n,-}R_j S j_{0,-}$ whereas $\rho_{\chi_0}(S) P_{n,-}$ is determined 
by the coset representatives $j_{n,-} R_j j_{0,-} S$ and $j_{0,-}S = S j_{0,-}$ we get $P_{n,-}\rho_{\chi_0}(S)=\rho_{\chi_0} (S) P_{n,-}$.  On the other hand  $T j_{0,-} = j_{0,-}T^{-1}$ and therefore  $P_{n,-}\rho_{\chi_0}(T)=\rho_{\chi_0}(T^{-1})P_{n,-}$ and hence the Theorem is proven.
\end{proof}
Obviously Theorem \ref{symmetry} follows now from Theorem \ref{Symm}. For the automorphisms $j_{n,+}= j_{n,-}j_{0,-}$ one gets the symmetry $\tilde{P}_{n,+}=\begin{pmatrix}P_{n,
+}&0\\0&P_{n,+}\end{pmatrix}$ with $P_{n,+}$ the permutation matrix corresponding to the permutation 
$\lambda_{n,-}\circ\sigma\circ \lambda_{0,-}\circ\sigma$ determined by the coset representatives $j_{n,+}R_j$.
\begin{remark}
The symmetry $P_{0,-}$ is given by $\rho_{\chi_0}(SM)$ where $M=\begin{pmatrix} 0&1\\1&0\end{pmatrix}$ and 
$\rho_{\chi_0}$ denotes the representation of $GL(2,\mathbb{Z})$ induced from the trivial character $\chi_0$ of 
$\Gamma_0(N)$. The transfer operator $\mathcal{L}^{MM}_\beta$ of Manin and Marcolli for $\Gamma_0(N)$ introduced in \cite{MM04} turns out to coincide with the 
operator $\rho_{\chi_0}(S) P_{0,-}\mathcal{L}_{\beta,\chi_0}^+\rho_{\chi_0}(S)$ and appears as a special case of our operators $P_{n,-}\mathcal{L}_{\beta,\chi_0}^+$.
\end{remark}
\begin{corollary}
The permutation matrices $P_{n,-}, \, 0\leq n\leq h_N-1,$ generate a finite group consisting of the permutation matrices $\{P_{n,\pm},\, 0\leq n\leq h_N-1\}$ and isomorphic to the normalizer group $\\mathcal(N)_N$ of $\Gamma_0(N)$ in $GL(2,\mathbb{Z})$. The symmetries $\{\tilde{P}_{n,\pm},\, 0\leq n\leq h_N-1\}$ of the transfer operator ${\bf{L}}_{\beta,\chi_0}$ for $\Gamma_0(N)$ and trivial character $\chi_0$ define a finite group isomorphic to the group $\mathcal{N}_N$.
\end{corollary}

\section{Selberg's character $\chi_\alpha$ for $\Gamma_0(4)$}
The group $\Gamma_0(4)$ is freely generated by the two elements $T=\begin{pmatrix}1&1\\0&1\end{pmatrix}$ and $B=\begin{pmatrix}1&0\\-4&1\end{pmatrix}$. Hence any $g\in \Gamma_0(4)$ can be written as $g=\prod\limits_{i=1}^{N_g} T^{m_i}B^{n_i}$.  If $\Omega(g)=\sum\limits_{i=1}^{N_g} m_i$ then Selberg's character  $\chi_\alpha$ \cite{Se89} is defined as 
\begin{equation}\label{Selb}
\chi_\alpha (g) = \exp (2\pi i \alpha \Omega (g) ),\; 0\leq \alpha\leq 1.
\end{equation}
Denote by   $z_i,\, 1\leq i\leq 3 $  the inequivalent cusps of $\Gamma_0(4)$ and by $T_i$ the generators of their 
stabilizer groups $\Gamma_{z_i}$ with $T_i z_i=z_i$.  They can be taken as $z_1=i \infty, z_2= 0, z_3= -\frac{1}{2}$ and 
$T_1 = T, T_2= B, T_3= T^{-1}B^{-1}$.
The character $\chi_\alpha$ is singular in the cusp $z_i$ iff $\chi_\alpha (T_i)=1$. Otherwise the character is non-singular 
in $z_i$. It is well known that the multiplicity  $\kappa(\chi_\alpha)$ of the continuous spectrum of the automorphic 
Laplacian $\Delta$ with character $\chi_\alpha$ is given by $\kappa(\chi_\alpha)=\#\{i: \chi_\alpha(T_i)=1\}$.  Therefore 
$\kappa(\chi_\alpha) = 3$ for $\alpha=0$ whereas $\kappa (\chi_\alpha)=1$ for $\alpha\not=0$ and hence the multiplicity of the continuos spectrum of the Laplacian changes from $3$ to $1$ when the trivial character is deformed to $\chi_\alpha $ with $\alpha\not= 0$. It is known \cite{PS85} that the character $\chi_\alpha$ is congruent (or arithmetic) iff $\alpha \in \{k\frac{1}{8},\; 0\leq k \leq 4\}$. Since the Selberg zeta function  given in (\ref{Selberg}) has the property $Z_{\Gamma_0(4),\chi_\alpha}=Z_{\Gamma_0(4),\chi_{-\alpha}}$ and obviously $\chi_\alpha=\chi_{\alpha+1}$ we can restrict the deformation parameter $\alpha$  to the range $0\leq \alpha\leq \frac{1}{2}$.
\begin{lemma}
The Selberg character $\chi_\alpha$ is invariant under the map $j_{2,-}z=\frac{z^*}{2z^*-1}$ and $J_{2,-}u(z):=u(j_{2,-}z)$ is a Maass form for $\Gamma_0(4)$ and character $\chi_\alpha$ if $u=u(z)$ is such a Maass form.
\end{lemma}
\begin{proof}
We only have to show that $\chi_\alpha$ is invariant under the map $j_{2,-}z=\frac{z^*}{2z^*-1}$. For $g=T$ we find $j_{2,-}T j_{2,-}=TB$ and hence $$\chi_\alpha (j_{2,-}T j_{2,-})=\chi_\alpha(TB)=\chi_\alpha(T),$$
whereas for $g=B$ one finds $j_{2,-}B j_{2,-}=B^{-1}$ and hence $$\chi_\alpha(j_{2,-}Bj_{2,-})=\chi_\alpha(B^{-1})=\chi_\alpha(B).$$
 Therefore $\chi_\alpha (j_{2,-}gj_{2,-})=\chi_\alpha(g)$ for all $g\in \Gamma_0(4)$. 
\end{proof}
For $u=u(z)$ a Maass form for $\Gamma_0(4)$ and character $\chi_\alpha$ and $\underline{\psi}=(\psi_j(\zeta))_{1\leq j\leq 6}$ its period function denote by $J_-u$ the Maass form $J_-u(z):= u(j_{2,-}z)$ respectively by $J_-\underline{\psi}=(J_-\psi_j(\zeta))_{1\leq j\leq 6}$ its period function. Then one shows
\begin{theorem}
The period function $J_-\underline{\psi}$ of the Maass form $J_-u$ is given by
\begin{equation}
J_-\psi_j(\zeta)= \zeta^{-2 \beta}\chi_\alpha(\eta_{\sigma\circ\delta(j)}) \psi_{\lambda_{2,-}\circ\sigma\circ\delta(j)}(\frac{1}{\zeta})
\end{equation}
where the permutations $\lambda_{2,-},\, \sigma,\, \delta$ respectively the  $\eta_j\in\Gamma_0(4)$ are determined through the coset representatives $R_j$ by $$j_{2,+}R_j=\theta_j R_{\lambda_{2,-}(j)},\, j_{0,-}R_jj_{0,-}= \gamma_j R_{\sigma(j)},\, R_jS=\eta_j R_{\delta(j)}$$ with $\theta_j,\, \gamma_j,\, \eta_j\in \Gamma_0(4)$ for $1\leq j\leq 6$.
\end{theorem}
\begin{proof}
Set $j_{\pm}:=j_{2,\pm}$ and $J_\pm u(z):=u(j_\pm z)$. Then $J_-u$ is a Maass form for $\Gamma_0(4)$ and character $\chi_\alpha $ whereas $J_+u$ is a Maass form for $\Gamma_0(4)$ and character $\chi_{-\alpha}$. The vector valued Maass form $J_+\underline{u}=(J_+u_j)_{1\leq j\leq 6}$ is given by $J_+u_j(z)=u(j_+R_jz)$. Therby we have choosen the  representatives $R_j$ of the cosets in $SL(2,\mathbb{Z})=\bigcup\limits_{1\leq j\leq 6} \Gamma_0(4) R_j$ as follows:
$$ R_1=id_2, R_j = ST^{j-2}, \, 2\leq j\leq 5, R_6= ST^2S.$$
But $j_+R_j=\theta_j R_{\lambda_{2,-}(j)}$ for some $\theta_j\in\Gamma_0(4)$ and some permutation $\lambda_{2,-}$ of the set $\{1,2,\ldots,6\}$ and hence $ J_+u_j(z)=\chi_\alpha(\theta_j) u(R_{\lambda_{2,-}(j)}z)$. It turns out that $\theta_j=B^{-1}$ for $1\leq j\leq 3$ and $\theta_j= id_2$ for $4\leq j\leq 6$. Hence $\chi_\alpha(\theta_j)=1$ and
\begin{equation}
J_+u_j(z)=u_{\lambda_{2,-}(j)}(z), \, 1 \leq j\leq 6,
\end{equation}
 with $ \lambda_{2,-} $ the permutation
\begin{equation} 
  \lambda_{2,-}=
	\begin{tabular}{cccccc}
	 1 & 2&3&4&5&6 \\
	 \hline 
	 6 & 4&5&2&3&1
	\end{tabular} .
\end{equation}
Consider next $J_+u_j(-z^*)=J_+u_j(j_{0,-}z)$. Then 
$$J_+u_j(j_{0,-}z)=u(j_+R_jj_{0,-}z)=u(j_+j_{0,-}j_{0,-}R_jj_{0,-}z).$$ 
If $j_{0,-}R_jj_{0,-}=\gamma_j R_{\sigma(j)}$ then $J_+u_j(j_{0,-}z)=u(j_-\gamma_jj_-j_-R_{\sigma(j)}z)$. But it turns out that $j_-\gamma_j j_-=id_2$ for $j=1,2,6$ respectively $j_-\gamma_j j_-=B$ for $j=3,4,5$, hence $\chi_\alpha(j_-\gamma_j j_-=1$ and therefore
\begin{equation}\label{-z^*}
J_+u_j(-z^*)=J_+u_j(j_{0,-}z)= J_-u_{\sigma(j)}(z).
\end{equation}
Since furthermore $J_+u_j(Sz)=u(j_+R_j Sz)=u(j_+\eta_j R_{\delta(j)}z)$ one finds
\begin{equation}\label{S}
J_+u_j(Sz)=\chi_{-\alpha}(\eta_j) u_{\lambda_{2,-}\circ\delta (j)}(z)
\end{equation}
where $\delta$ is the following permutation of the set $\{1,2,\ldots,6\}$ 
 \begin{equation} 
  \delta=
	\begin{tabular}{cccccc}
	 1 & 2&3&4&5&6 \\
	 \hline 
	 2 &1&5&6&3&4
	\end{tabular} .
\end{equation}
and $\eta_j= id_2$ for $j=1,2,4,6$  respectively $\eta_3=\eta_5^{-1}=T^{-1}B^{-1}$.
For $J_+u_j(-Sz^*)$ one gets with (\ref{-z^*}) $J_+u_j(-Sz^*)=J_-u_{\sigma(j)}(Sz)=u(j_{2,-}R_{\sigma(j)}Sz)$ and hence\newline
$J_+u_j(-Sz^*)=u(j_{2,-}\eta_{\sigma(j)}R_{\delta\circ\sigma(j)}z)=\chi_\alpha(\eta_{\sigma(j)})J_-u_{\delta\circ\sigma(j)}(z)$. Using the explicit form of the $\eta_j$ one shows $\chi_\alpha(\eta_{\sigma(j)})=\chi_{-\alpha}(\eta_j)$ and therefore
\begin{equation}\label{-Sz^*}
J_+u_j(-Sz^*)=\chi_{-\alpha}(\eta_j) J_-u_{\delta\circ\sigma(j)}(z).
\end{equation}
Define next $v_{\pm,j}=v_{\pm,j}(z)$ as \begin{equation}\label{v}
v_{\pm,j}(z):= J_+u_j(z) \pm J_+u_j(-z^*).\end{equation} 
Then $
v_{\pm,j}(-z^*)=\pm v_{\pm,j}(z)$ and by (\ref{S}) respectively (\ref{-Sz^*})  \begin{equation}\label{v'}
v_{\pm,j}(Sz)=\chi_{-\alpha}(\eta_j) v_{\pm,\delta(j)}(z)\end{equation}
If therefore $\psi'_{\pm, j}(\zeta):= \int\limits_0^{i \infty}\eta(v_{\pm,j},R_\zeta^\beta)(z)$ one gets  from relation (\ref{v'})
\begin{equation}
\psi'_{\pm, j}(\zeta)=\pm \zeta^{-2\beta}\chi_{-\alpha}(\eta_j) \psi'_{\pm, \delta(j)}(\frac{1}{\zeta})
\end{equation}   
and using the identity (\ref{v})
\begin{equation}
\psi_{\lambda_{2,-}(j)}(\zeta)\pm J_-\psi_{\sigma(j)}(\zeta)=\pm \zeta^{-2\beta}\chi_{-\alpha}(\eta_j)\left(\psi_{\lambda_{2,-}\circ\delta(j)}(\frac{1}{\zeta})\pm J_-\psi_{\sigma\circ\delta(j)}(\frac{1}{\zeta})\right).
\end{equation}
Adding these two equations leads finally to
\begin{equation}
J_-\psi_j(\zeta)=\zeta^{-2\beta} \chi_\alpha(\eta_{\sigma\circ\delta(j)})\psi_{\lambda_{2,-}\circ\sigma\circ\delta(j)}(\frac{1}{\zeta}).
\end{equation}
\end{proof}
Inserting the explicit form of the permutations  
\begin{equation}\label{permutation} 
 \sigma\circ\delta=
	\begin{tabular}{cccccc}
	 1 & 2&3&4&5&6 \\
	 \hline 
	 2 &1&3&6&5&4
	\end{tabular} \end{equation}
respectively \begin{equation}
\lambda_{2,-}\circ\sigma\circ  \delta=
	\begin{tabular}{cccccc}
	 1 & 2&3&4&5&6 \\
	 \hline 
	 4 &6&5&1&3&2
	\end{tabular} 
	\end{equation}
	and the character values
	
$$\chi_\alpha(\eta_1)=\chi_\alpha(\eta_2)=\chi_\alpha(\eta_4)=\chi_\alpha(\eta_6)=1$$
respectively
$$\chi_\alpha(\eta_3)=\chi_\alpha(\eta_5)^{-1}=	e^{-2\pi i \alpha}$$
one finds
\begin{eqnarray}
J_-\psi_1(\zeta)&=&\zeta^{-2\beta} \psi_4(\frac{1}{\zeta})\nonumber\\
J_-\psi_2(\zeta)&=&\zeta^{-2\beta} \psi_6(\frac{1}{\zeta})\nonumber\\
J_-\psi_3(\zeta)&=&\zeta^{-2\beta} e^{-2\pi i \alpha}\psi_4(\frac{1}{\zeta})\\
J_-\psi_4(\zeta)&=&\zeta^{-2\beta} \psi_1(\frac{1}{\zeta})\nonumber\\
J_-\psi_5(\zeta)&=&\zeta^{-2\beta} e^{2\pi i \alpha}\psi_3(\frac{1}{\zeta})\nonumber\\
J_-\psi_6(\zeta)&=&\zeta^{-2\beta} \psi_2(\frac{1}{\zeta})\nonumber
\end{eqnarray}	
Define the matrix $Q_{2,-}$ through the equation $J_{2,-}\underline{\psi}(\zeta)= \zeta^{-2\beta} Q_{2,-}\underline{\psi}(\frac{1}{\zeta})$. Then one gets
\begin{proposition}
The permutation matrix $P_{2,-}:=\rho_{\chi_\alpha}(S) Q_{2,-}$ defines a symmetry $\tilde{P}_{2,-}=\begin{pmatrix}0&P_{2,-}\\P_{2,-}&0\end{pmatrix}$ of the transfer operator $${\bf{L}}_{\beta,
\chi_\alpha}=\begin{pmatrix}0&\mathcal{L}_{\beta,\chi_\alpha}^+\\\mathcal{L}_{\beta,\chi_\alpha}^+&0\end{pmatrix}$$ for $\Gamma_0(4)$ and character $\chi_\alpha$ with 
$P_{2,-}^2=id_6$ and $P_{2,-} \rho_{\chi_\alpha} (S)= \rho_{\chi_\alpha} (S) P_{2,-}$
 respectively $P_{2,-} \rho_{\chi_\alpha} (T)= \rho_{\chi_\alpha} (T^{-1}) P_{2,-}$ and therefore $P_{2,-} \mathcal{L}_{\beta,\chi_\alpha}^+=\mathcal{L}_{\beta,\chi_\alpha}^-  P_{2,-}$. The permutation 
matrix $P_{2,-}$ corresponds to the permutation $\lambda_{2,-}\circ \sigma$ and hence is determined by the coset representatives $J_{2,-}R_j j_{0,-}$.  
\end{proposition}	
\begin{proof}
For our choice of coset representatives $R_j$ as given in (\ref{Repres}) one finds for $\rho_{\chi_\alpha}(S)$

\begin{equation}\label{rhoS}
\rho_{\chi_\alpha}(S)=\begin{pmatrix}
0&1&0&0&0&0\\
1&0&0&0&0&0\\
0&0&0&0&e^{-2\pi i \alpha}&0\\
0&0&0&0&0&1\\
0&0&e^{2\pi i \alpha}&0&0&0\\
0&0&0&1&0&0
\end{pmatrix},\end{equation}
and hence the matrix $Q_{2,-}$ is given by
 \begin{equation}Q_{2,-}=\begin{pmatrix}
0&0&0&1&0&0\\
0&0&0&0&0&1\\
0&0&0&0&e^{-2\pi i \alpha}&0\\
1&0&0&0&0&0\\
0&0&e^{2\pi i \alpha}&0&0&0\\
0&1&0&0&0&0
\end{pmatrix}.\end{equation}
For $\rho_{\chi_\alpha}(T)$ one finds
\begin{equation}\label{rhoT}
\rho_{\chi_\alpha}(T)=\begin{pmatrix}
e^{2\pi i \alpha}&0&0&0&0&0\\
0&0&1&0&0&0\\
0&0&0&1&0&0\\
0&0&0&0&1&0\\
0&1&0&0&0&0\\
0&0&0&0&0&e^{-2\pi i \alpha}
\end{pmatrix}.\end{equation}
A simple calculation then confirms that $P_{2,-} \rho_{\chi_\alpha} (S)= \rho_{\chi_\alpha} (S) P_{2,-}$ respectively $P_{2,-} \rho_{\chi_\alpha} (T)= \rho_{\chi_\alpha} (T^{-1}) P_{2,-}$ with 
\begin{equation}\label{P}
P_{2,-}=\begin{pmatrix}
0&0&0&0&0&1\\
0&0&0&1&0&0\\
0&0&1&0&0&0\\
0&1&0&0&0&0\\
0&0&0&0&1&0\\
1&0&0&0&0&0
\end{pmatrix}\end{equation}
and hence defines a symmetry of the transfer operator ${\bf{L}}_{\beta,\chi_\alpha}$.
 The matrix $P_{2,-}$ coincides with the permutation matrix $P_2$  corresponding to the permutation $\sigma_2$ in (\ref{sym2}). 
\end{proof}
\begin{remark}
For the trivial character $\chi_0$ also the map $j_{0,-}z=-z^*$ defines an automorphism of the Maass forms for the group 
$\Gamma_0(4)$, indeed this is an automorphism for all Hecke congruence subgroups $\Gamma_0(N)$. In this case the 
permutation $\lambda_{0,-}$ is the trivial permutation and the matrix $Q_{0,-}$ is determined by the permutation 
$\sigma\circ\delta$. For $\Gamma_0(4)$ this is given by (\ref{permutation}). Using (\ref{rhoS}) with $\alpha = 0$ one 
obtains for  $P_{0,-}=\rho_{\chi_0}(S) Q_{0,-}$ just the permutation $\sigma_1$ as given in (\ref{sym1}). The symmetry $\tilde{P}_1$ for $\Gamma_0(4)$ and trivial character $\chi_0$ hence corresponds to the automorphism $z\to -z^*$ of the Maass forms for this group.
\end{remark}
We have seen that for every eigenfunction $\underline{f}=\underline{f}(\zeta)$ of the operator $P_2\mathcal{L}_{\beta,\chi_{\alpha}}^+$ with eigenvalue $\lambda =\pm 1$ the function $\underline{\Psi}=\underline{\Psi}(\zeta)= P_2\rho_{\chi_{\alpha}}(T^{-1}S)P_2\underline{f}(\zeta -1)$ fulfills the functional equation
\begin{equation}
\underline{\Psi}(\zeta)=\lambda \zeta^{-2 \beta} \rho_{\chi_{\alpha}}(S) P_2 \underline{\Psi}(\frac{1}{\zeta})=\lambda J_-\underline{\Psi}(\zeta)
\end{equation}
and hence is an eigenfunction of the involution $J_-$ corresponding to the automorphism $j_-=j_{2,-}$ of the Maass forms for $\Gamma_0(4)$ and character $\chi_\alpha$. Hence this shows
\begin{proposition}
The eigenfunctions $\underline{f}=\underline{f}(\zeta)$ of the operator $P_2\mathcal{L}_{\beta,\chi_{\alpha}}^+$ with eigenvalue $\lambda=\pm 1$ correspond to Maass forms which are even respectively odd under the involution $J_-=J_{2,-}.$
\end{proposition}
Phillips and Sarnak have shown in \cite{PS85} for a conjugate character $\hat{\chi}_\alpha$ that the Maass cusp forms
odd under the corresponding conjugate involution $\hat{J}$ stay cusp forms under the deformation of this character. Hence we get as a corollary of their result
\begin{corollary}
The zero's of the Selberg zeta function for the group $\Gamma_0(4)$ and character $\chi_\alpha$  corresponding to eigenfunctions of the operator $P_2\mathcal{L}_{\beta,\chi_{\alpha}}^+$ with eigenvalue $\lambda=-1$ which for $\alpha=0$ are on the critical line $\Re \beta =\frac{1}{2}$ stay for all $\alpha$ on this line.
\end{corollary}
\begin{remark}
The operator $P_2\mathcal{L}_{\beta,\chi_{\alpha}}^+$ can be used to calculate numerically the Selberg zeta function for small values of $\Im \beta$ and arbitrary $0\leq \alpha\leq \frac{1}{2}$. These numerical calculations confirm the above  Corollary and let us expect that all the zero's of the Selberg function corresponding to the eigenvalue $\lambda=1$ of the operator $P_2\mathcal{L}_{\beta,\chi_{\alpha}}^+$ for $\alpha=0$ leave the critical line when $\alpha$ becomes positive. A detailed discussion of the numerical treatment of the behaviour of the zero's of Selberg's function under the character deformation will appear elsewhere \cite{BFM10}. 
\end{remark}

\bibliographystyle{amsalpha}

\begin{thebibliography}{99999}

\raggedright

\bibitem{BFM10}
R.W.\ Bruggeman, M.\ Fraczek and D.\ Mayer,
\newblock \textit{Perturbation of Zeros of the Selberg Zeta Function for $\Gamma_0(4)$},
\newblock in preparation.


\bibitem{CM00}
C.H.\ Chang and D.\ Mayer,
\newblock \textit{Thermodynamic formalism and Selberg's zeta function for modular groups}.
\newblock Regul.\ Chaotic Dyn. \textbf{5} (2000), no.3, 281--312. \newline
\newblock \href{http://www.ams.org/mathscinet-getitem?mr=1789478}{\texttt{\small MR1789478}},
          \href{doi:10.1070/rd2000v005n03ABEH000150	http://dx.doi.org/10.1070/rd2000v005n03ABEH000150}{\texttt{\small doi:10.1070/rd2000v005n03ABEH000150}}

\bibitem{CM01}
C.H.\ Chang and D.\ Mayer,
\newblock \textit{An extension of the thermodynamic formalism approach to Selberg's zeta function for general modular groups}.
\newblock in \textit{Ergodic Theory, Analysis, and Efficient Simulation of Dynamical Systems} (ed. B. Fiedler).
\newblock Springer Verlag, Berlin, (2001) 523--562. \newline
\newblock \href{http://www.ams.org/mathscinet-getitem?mr=1850321}{\texttt{\small MR1850321}}

\bibitem{CM01a}
C.H.\ Chang and D.\ Mayer,
\newblock \textit{Eigenfunctions of the transfer operators and the period functions for modular groups}.
\newblock in \textit{Dynamical, spectral, and arithmetic zeta functions} (San Antonio, TX, 1999).
\newblock Contemp.\ Math.\ \textbf{290} (2001), 1--40. \newline
\newblock \href{http://www.ams.org/mathscinet-getitem?mr=1868466}{\texttt{\small MR1868466}}

\bibitem{CM98}
C.H.\ Chang and D.\ Mayer,
\newblock \textit{The period function of the nonholomorphic Eisenstein series for $PSL(2,\mathbb{Z})$}.
\newblock Math. Phys. Electronic J.\ \textbf{4} (1998), Paper 6, 8 pp. \newline
\newblock \href{http://www.ams.org/mathscinet-getitem?mr=1647288}{\texttt{\small MR1647288}}

\bibitem{CN79}
J.H.\ Conway and S.P.\ Norton,
\newblock \textit{Monstrous moonshine}.
\newblock Bull. London Math. Soc. \textbf{11} (1979),  308--339. \newline
\newblock \href{http://www.ams.org/mathscinet-getitem?mr=0554399}{\texttt{\small MR0554399}},
 \href{http://dx.doi.org/10.1112/blms/11.3.308}{\texttt{\small doi:10.1112/blms/11.3.308}}	         

\bibitem{DH07}
A.\ Deitmar and J.\ Hilgert,
\newblock \textit{A Lewis correspondence for submodular groups}.
\newblock Forum Math.\ \textbf{19} (2007), 1075--1099. \newline
\newblock \href{http://www.ams.org/mathscinet-getitem?mr=2367955}{\texttt{\small MR2367955}}
\href{http://dx.doi.org/10.1515/F0RUM.2007.042}{\texttt{\small doi:10.1515/F0RUM.2007.042}}	

\bibitem{E93}
 I.\ Efrat,
 \newblock \textit{Dynamics of the continued fraction map and the spectral theory of $SL(2,\mathbb{Z})$}.
 \newblock Invent. math \textbf{114} (1993), 207--218. \newline
\newblock \href{http://www.ams.org/mathscinet-getitem?mr=2468476}{\texttt{\small MR1235024}},
 \href{http://dx.doi.org/10.1007/BF01232667}{\texttt{\small doi:10.1007/BF01232667}}	

\bibitem{F10}
M.\ Fraczek, 
\newblock \textit{Character deformation of the SElberg zeta function for congruence subgroups via the transfer operator}.
\newblock PhD thesis, Clausthal Institute of Technology (2010)


\bibitem{H83}
 D.\ Hejhal,
\newblock \textit{The Selberg Trace Formula for $PSL(2,\mathbb{R})$}.
\newblock Lecture Notes in Mathematics \textbf{1001} Springer Verlag (1983), ch. X, \S 2, \S 5. \newline
\newblock \href{http://www.ams.org/mathscinet-getitem?mr=0711197}{\texttt{\small MR0711197}},
  

\bibitem{LN64}
J.\ Lehner and M.\ Newman,
\newblock \textit{Weierstrass points of $\Gamma_0(n)$}.
\newblock Ann.\ of Math. \textbf{79} no.2 (1964), 360--368. \newline
\newblock \href{http://www.ams.org/mathscinet-getitem?mr=0161841}{\texttt{\small MR0161841}}

\bibitem{LZ01}
J.\ Lewis and D.\ Zagier,
\newblock \textit{Period functions for Maass wave forms.I.}.
\newblock Ann.\ of Math. \textbf{153} (2001), 191--258. \newline
\newblock \href{http://www.ams.org/mathscinet-getitem?mr=1826413}{\texttt{\small MR1826413}},
          \href{http://dx.doi.org/10.2307/2661374}{\texttt{\small doi:10.2307/2661374}}

\bibitem{MM04}
Y.\ Manin and M.\ Marcolli,
\newblock \textit{Continued fractions, modular symbols, and noncommutative geometry}.
\newblock  Selecta Math. (N.S.) \textbf{8} (2002), no. 3, 475--521. \newline
\newblock \href{http://www.ams.org/mathscinet-getitem?mr=1826413}{\texttt{\small MR1931172}},
\href{http://dx.doi.org/10.1007/s00029-002-8113-3}{\texttt{\small doi:10.1007/s00029-002-8113-3}}

          \href{http://projecteuclid.org/euclid.cmp/1104200514}{\texttt{\small euclid.cmp/1104200514}}
	


\bibitem{MM10}
D.\ Mayer and T.\ M\"uhlenbruch,
\newblock \textit{Nearest $\lambda_q$-multiple fractions}.
\newblock  Proceedings of the CRM Spectrum and Dynamics Workshop, Montreal 2008, CRM Proceedings and Lecture Notes Series, \textbf{52} (2010), 147--184 . \newline
\newblock \href{http://arxiv.org/abs/0902.3953}{\texttt{\small arXiv:0902.3953}}

\bibitem{MMS11}
D.\ Mayer , T.\ M\"uhlenbruch and F.\ Str\"omberg,
\newblock \textit{The transfer operator for the Hecke triangle groups}.
\newblock  Proceedings of International Conference Dynamical Systems II, Denton, UNT 2009 \newline
\newblock To appear in  Discrete and Continuous Dynamical  Systems (2011).
\newblock \href{http://arxiv.org/abs/09122236}{\texttt{\small arXiv:0902.3953}}



\bibitem{MS08}
D.\ Mayer and F.\ Str\"omberg,
\newblock \textit{Symbolic dynamics for the geodesic flow on Hecke surfaces}.
\newblock Journal of Modern Dynamics \textbf{2} (2008), 581--627. \newline
\newblock \href{http://www.ams.org/mathscinet-getitem?mr=2449139}{\texttt{\small MR2449139}},
          \href{http://dx.doi.org/10.3934/jmd.2008.2.581}{\texttt{\small doi:10.3934/jmd.2008.2.581}}

\bibitem{Mu06}
 T.\ M\"uhlenbruch,
\newblock \textit{Hecke operators on period functions for $\Gamma_0(N)$}.
\newblock J. Number Th. \textbf{118} (2006), 208--235.  \newline
\newblock \href{http://arxiv.org/abs/2225281}{\texttt{\small arXiv:2225281}}
\href{	http://dx.doi.org/10.1016/j.jnt.2005.09.003}{\texttt{\small doi:10.1016/j.jnt.2005.09.003}}

\bibitem{PS85}
R.\ Phillips and P.\ Sarnak,
\newblock \textit{On cusp forms in character varieties},
\newblock Geometric and Functional Anaysis\ \textbf{4} (1994), 93--118.\newline
\newblock \href{http://www.ams.org/mathscinet-getitem?mr=1254311}{\texttt{\small MR1254311}},
          \href{	http://dx.doi.org/10.1007/BF01898362}{\texttt{\small doi:10.1007/BF01898362}}

\bibitem{Se89}
A.\ Selberg,
\newblock \textit{Remarks on the distribution of poles of Eisenstein series},
\newblock in \textit{Collected Papers}, Vol.\ 2, 15--46, Springer Verlag, 1989. \newline
\newblock \href{http://www.ams.org/mathscinet-getitem?mr=1295844}{\texttt{\small MR1295844}}

\bibitem{St08}
F.\ Str\"omberg,
\newblock \textit{Computation of Selberg's zeta functions on Hecke triangle groups}. \newline
\newblock \href{http://arxiv.org/abs/0804.4837}{\texttt{\small arXiv:0804.4837}}

\bibitem{V90}
A.\ Venkov,
\newblock \textit{Spectral Theory of Automorphic Functions}. \newline
\newblock Mathematics and its Applications (Soviet Series), ch. 7, Kluwer, Dordrecht (1990).
\newblock \href{http://www.ams.org/mathscinet-getitem?mr=1135112}{\texttt{\small MR1135112}}

\bibitem{Z02}
D.\ Zagier,
\newblock \textit{New points of view on the Selberg zeta function},
\newblock in \textit{Proceedings of Japanese-German Seminar "Explicit structures of modular forms and zeta functions", Hakuba,  (2002), 1--10, Ryushi-do}.

\end{thebibliography}


\end{document}